\documentclass{article}

\usepackage{amsfonts,amsmath,amssymb}
\usepackage{graphicx,epsfig}
\usepackage{float}
\usepackage{subfigure}
\usepackage{multirow}
\usepackage{verbatim}

\usepackage[usenames,dvipsnames]{color}

\usepackage[
          hypertex,
          ,colorlinks=true,linkcolor=blue
          ,pdftex
          ,pdfproducer={Latex with hyperref}
          ,pdfcreator={pdflatex}
          ]{hyperref}

\newtheorem{theorem}{Theorem}

\newtheorem{lemma}[theorem]{Lemma}

\newtheorem{remark}[theorem]{Remark}

\newenvironment{proof}[1][Proof]{\noindent\textbf{#1.} }{\ \rule{0.5em}{0.5em}}
\newdimen\dummy
\dummy=\oddsidemargin
\def\MM{\mathbb{M}}

\def\RR{\mathbb{R}}

\def\be{\begin{equation}}
\def\ee{\end{equation}}

\let\ds\displaystyle
\let\eps\varepsilon

\begin{document}

\title{An Asymptotic-Preserving method for highly anisotropic elliptic
  equations based on a micro-macro decomposition}

\author{
  Pierre Degond\footnotemark[2]\ \footnotemark[3] 
  \and Alexei Lozinski\footnotemark[2]
  \and Jacek Narski\footnotemark[2] 
  \and Claudia Negulescu\footnotemark[4] } 

\renewcommand{\thefootnote}{\fnsymbol{footnote}}

\footnotetext[2]{Universit\'e de Toulouse, UPS, INSA, UT1, UTM, Institut de Math\'ematiques de Toulouse, F-31062 Toulouse, France}
\footnotetext[3]{CNRS, Institut de Math\'ematiques de Toulouse UMR 5219, F-31062 Toulouse, France}
\footnotetext[4]{CMI/LATP, Universit\'e de Provence, 39 rue Fr\'ed\'eric Joliot-Curie 13453 Marseille cedex 13}

\renewcommand{\thefootnote}{\arabic{footnote}}
\maketitle

\abstract{The concern of the present work is the introduction of a
  very efficient Asymptotic Preserving scheme for the resolution of
  highly anisotropic diffusion equations. The characteristic features
  of this scheme are the uniform convergence with respect to the
  anisotropy parameter $0<\eps <<1$, the applicability (on cartesian
  grids) to cases of non-uniform and non-aligned anisotropy fields $b$
  and the simple extension to the case of a non-constant anisotropy
  intensity $1/\eps$. The mathematical approach and the numerical
  scheme are different from those presented in the previous work
  [Degond et al. (2010), arXiv:1008.3405v1] and its considerable
  advantages are pointed out.}
\section{Introduction}\label{SEC1}

The numerical resolution of highly anisotropic physical problems is a
challenging task. In particular, it is difficult to capture the
behaviour of physical phenomena characterized by strong anisotropic
features since a straight-forward discretization leads typically to
very ill-conditioned problems. In the class of problems addressed in
this paper the anisotropy is aligned with a vector field which may be
variable in space/time. Such problems are encountered in many physical
applications, for example flows in porous media \cite{porous1,TomHou},
semiconductor modeling \cite{semicond}, quasi-neutral plasma
simulations \cite{Navoret}, the list of possible applications being
not exhaustive. The motivation of this work is closely related to the
magnetized plasma simulations such as atmospheric plasmas
\cite{Kelley2,Kes_Oss}, internal fusion plasmas \cite{Beer,Sangam} or
plasma thrusters \cite{SPT}. In this case the anisotropy direction is
defined by a magnetic field confining the particles around the field
lines, the anisotropy intensity $1/\eps$ reaching orders of magnitude
as high as $10^{10}$.  The difficulty with these anisotropic problems
is that they become singular in the limit $\eps \rightarrow 0$, where
$\eps$ is a small parameter responsible for the strong anisotropy of
the problem. A straightforward discretization of such problems (using,
for example, finite difference methods) results in the inversion of a
very badly conditioned linear system, which becomes unfeasible for
$\eps <<1$.  In this paper we present a new approach based on the so
called Asymptotic Preserving reformulation introduced initially in
\cite{ShiJin}. This work presents an important improvement to the
method presented in the previous paper \cite{DDLNN} and is aimed to
treat in a precise manner a variety of strong anisotropies with low
numerical costs.

The model problem we are interested in,  reads
\begin{gather}
  \left\{ 
    \begin{array}{ll}
      - \nabla \cdot \mathbb A \nabla u^{\varepsilon } = f       & \text{ in }
      \Omega, \\[3mm]
      n \cdot \mathbb A \nabla u^{\varepsilon }= 0
      & \text{ on } \Gamma_N\,,\\[3mm]
      u^{\varepsilon }= 0
      & \text{ on } \Gamma_D\,,
    \end{array}
  \right.
  \label{InitP}
\end{gather}
where $\Omega \subset \RR^{2}$ or $\Omega \subset \RR^{3}$ is a
bounded domain with boundary $\partial \Omega = \Gamma_D \cup
\Gamma_N$ and outward normal $n$. The direction of the anisotropy is
given by a vector field $B$, where we suppose $\text{div} B = 0$ and
$B \neq 0$. The direction of $B$ shall be denoted by the unit vector field
$b = B/|B|$. The domain boundary is decomposed into $\Gamma_D:= \{ x
\in \partial\Omega \ | \ b (x) \cdot n = 0 \}$
and $\Gamma_N:= \partial \Omega \backslash \Gamma_D$.  The
anisotropic diffusion matrix is then given by
\begin{gather} \mathbb A = \frac{1}{\varepsilon }
  A_\parallel b \otimes b + (Id - b \otimes b)A_\perp (Id - b \otimes
  b)\,.
  \label{eq:Jh0a}
\end{gather}
The scalar field $A_\parallel>0$ and the symmetric positive definite
matrix field $A_\perp$ are of order one while the parameter $0 <
\varepsilon < 1$ can be very small, provoking thus the high anisotropy
of the problem.  The system becomes ill posed if we consider the
formal limit $\eps \rightarrow 0$. It is thus very ill conditioned for
$\eps << 1$. The goal of the present paper is to circumvent this
difficulty and to propose a numerical scheme which is uniformly
convergent with respect to the parameter $\eps$.

This model problem has been studied before in the Asymptotic
Preserving context, see for example \cite{DDN,besse,DDLNN}. The key
idea was to decompose the solution $u^\eps$ into two parts: $p^\eps$,
which is constant along the anisotropy direction and $q^\eps$, which
contains the fluctuating part. The resulting modified linear system is
bigger than the original problem but has the important feature of
reducing, as $\eps \rightarrow 0$, to the so-called {\it Limit} model
(L-model), which is a well-posed system satisfied by the limit
solution $u^0= \lim_{\eps \rightarrow 0} u^\eps$. In other words, this
AP-procedure transforms a singularly perturbed problem in an
equivalent regularly perturbed one.

In \cite{DDN}, a special case of an anisotropy aligned with the
$z$-axis was studied. The Asymptotic Preserving reformulation was
obtained in the following way. Firstly, the original problem was
integrated along the anisotropy direction, resulting in an
$\eps$-independent elliptic equation for the mean part $p^\eps$. Then,
the mean equation was subtracted from the original problem giving rise
to an elliptic equation for the fluctuating part $q^\eps$. A
generalization of this approach was proposed in \cite{besse}, in the
framework of curvilinear anisotropy fields $b$. The introduction of an
adapted curvilinear coordinate system with one coordinate aligned with
the anisotropy direction allowed to reduce the problem to the one
studied in \cite{DDN} and hence to address more realistic problems, as
for example the ionospheric plasma simulations. 

In \cite{DDLNN}, the original method of \cite{DDN} is generalized for
arbitrary fields $b$ in a different manner. Instead of performing a
coordinate-system transformation and integrating along the anisotropy
direction, the new technique uses a Cartesian grid in combination with
an adapted mathematical framework. The mean part $p^\eps$ is forced to
be in the space of functions constant along the $b$ field by means of
a Lagrange multiplier technique. Similarly, the fluctuating part
$q^\eps$ is forced to be in the orthogonal space (with respect to the
$L^2$-norm). This approach requires the introduction of three Lagrange
multipliers and results in a system with five unknowns. The advantage
of this method lies in its generality. Since no change of the
coordinate system is needed, no computation of the field lines, no
integration along them as well as no mesh adaptation is required, this
numerical method is easy to implement even for anisotropies varying in
time. The method works well on Cartesian grids even for strongly
varying $b$ fields. We will refer to the method of \cite{DDLNN} as the
Duality-Based (DB) Asymptotic Preserving reformulation or simply as
the DB-method, to emphasize the extensive use of Lagrange multipliers
in its derivation.

In the present paper we introduce a novel approach. It is as general
as the DB-method, i.e. no change of coordinate system nor mesh
adaptation is required. It is however much more efficient in terms of
computational cost. The new idea arises from the following questions:
does $p^\eps$ need to be the average of $u^\eps$ along the anisotropy
direction? Does $q^\eps$ need to be orthogonal to $p^\eps$, with
respect to the $L^2$-norm? Is $u^\eps = p^\eps + q^\eps$ an optimal
decomposition?  In this paper we exploit a different decomposition:
$u^\eps = p^\eps + \eps q^\eps$, with $q^\eps$ not being orthogonal to
$p^\eps$. Instead, $q^\eps$ belongs to the space of functions
vanishing on the ``inflow" part of the boundary, i.e. where $b\cdot
n<0$.  This space is easy to discretize without introducing any
additional Lagrange multiplier.  As a result, $p^\eps$ is no longer an
average of $u^\eps$, however we still choose it to be constant along
the anisotropy direction.  The AP reformulation of the original
equation is then obtained in terms of $u^\eps$ and $q^\eps$. We use
first the fact that the only component of $u^\eps$ that varies along
the anisotropy direction is $\eps q^\eps$, and thus replace
$\nabla_{||} u^\eps$ by $\eps \nabla_{||} q^\eps$ in the term of order
$1/\eps$. We obtain thus an equation whose coefficients are all of
order one. We should add also a second equation assuring than $p^\eps
=u^\eps - \eps q^\eps$ is indeed constant along the field lines.  The
resulting system consists of only two equations compared to the five
equations of the DB-method. It should be noted that, while keeping the
advantages of the DB-method (no need of mesh- or
coordinate-system-adaptation with respect to the anisotropy direction,
uniform convergence with respect to $\eps$), the new method is far
more efficient in terms of memory requirements as well as
computational time. The new method will be referred to as the
Micro-Macro (MM) Asymptotic Preserving reformulation or simply as the
MM-method, to emphasize the fact that the fluctuating part $q^\eps$ of
the solution is now rescaled by $\eps$. We will prove, in particular,
that $q^\eps$ does not explode as $\eps\to 0$, so that the ansatz
$u^\eps = p^\eps + \eps q^\eps$ does indeed introduce a proper
rescaling for the ``micro"-fluctuations over the ``macro" part of the
solution contained in $p^\eps$.

The paper is organized as follows. Section \ref{SEC2} introduces the
Singularly Perturbed elliptic problem (P-model) and the new Asymptotic
Preserving reformulation (MM-method). We also perform a mathematical
study of this reformulation, in particular, of the convergence towards
the Limit model, as $\eps \rightarrow 0$. In Section \ref{SEC4} we
present the discretization and the numerical results of various test
cases for constant and variable anisotropy fields $b$. We compare the
efficiency of the new MM-method with the previously proposed
DB-method. Finally, in Section \ref{GENE} we propose a generalization
of the MM-method to the case of a non-constant anisotropy parameter
$\eps$. The detailed numerical analysis of the MM-method is postponed
to a forthcoming work.

\section{The mathematical problem} \label{SEC2}
Let $b$
be a smooth field in a domain $\Omega \subset \RR^d$, with $d=2,3$,
and let us decompose the boundary $%
\Gamma =\partial \Omega $ into three components following the sign of
the intersection with $b$:
$$
\Gamma _{D}:= \{ x \in \Gamma \,\, / \,\, b(x) \cdot n(x) =0 \}\,,
\quad \Gamma _{in}:= \{ x \in \Gamma \,\, / \,\, b(x) \cdot n(x) <0
\}\,, \quad \Gamma _{out}:= \{ x \in \Gamma \,\, / \,\, b(x) \cdot
n(x) > 0 \}\, .
$$
The vector $n$ is here the unit outward normal on $\Gamma $. We denote
$\Gamma _{N}:=\Gamma _{in}\cup \Gamma _{out}$. We also assume $b \in
(C^{\infty}(\bar\Omega))^d$ and $|b(x)|=1$ for all $x \in \bar\Omega$.

Given this vector field $b$, one can decompose now vectors
$v \in \mathbb R^d$, gradients $\nabla \phi$, with $\phi(x)$ a scalar
function, and divergences $\nabla \cdot v$, with $v(x)$ a vector
field, into a part parallel to the anisotropy direction and a part
perpendicular to it.  These parts are defined as follows:
\begin{equation} 
  \begin{array}{llll}
    \ds v_{||}:= (v \cdot b) b \,, & \ds v_{\perp}:= (Id- b \otimes b) v\,, &\textrm{such that}&\ds
    v=v_{||}+v_{\perp}\,,\\[3mm]
    \ds \nabla_{||} \phi:= (b \cdot \nabla \phi) b \,, & \ds
    \nabla_{\perp} \phi:= (Id- b \otimes b) \nabla \phi\,, &\textrm{such that}&\ds
    \nabla \phi=\nabla_{||}\phi+\nabla_{\perp}\phi\,,\\[3mm]
    \ds \nabla_{||} \cdot v:= \nabla \cdot v_{||}  \,, & \ds
    \nabla_{\perp} \cdot v:= \nabla \cdot v_{\perp}\,, &\textrm{such that}&\ds
    \nabla \cdot v=\nabla_{||}\cdot v+\nabla_{\perp}\cdot v\,,
  \end{array}
\end{equation}
where we denoted by $\otimes$ the vector tensor product. With these
notations we can now introduce the mathematical problem, the so-called
Singular Perturbation problem, whose numerical solution is the main
concern of this paper.

\subsection{The Singular-Perturbation problem (P-model)} \label{SEC21}
Our starting problem is the following Singular-Perturbation problem (P-problem)
\begin{gather}
    (P)\,\,\,
  \left\{
    \begin{array}{ll}
      -{1 \over \varepsilon} \nabla_\parallel \cdot 
      \left(A_\parallel \nabla_\parallel u^{\varepsilon }\right) 
      - \nabla_\perp \cdot 
      \left(A_\perp \nabla_\perp u^{\varepsilon }\right) 
      = f 
      & \text{ in } \Omega, \\[3mm]
      {1 \over \varepsilon} 
      n_\parallel \cdot 
      \left( A_\parallel \nabla_\parallel u^{\varepsilon } \right)
      +
      n_\perp \cdot 
      \left(A_\perp \nabla_\perp u^{\varepsilon }\right) 
      = 0
      & \text{ on } \Gamma_N,  \\[3mm]
      u^{\varepsilon }= 0
      & \text{ on } \Gamma_D\,.
    \end{array}
  \right.
  \label{P} 
\end{gather}
The diffusion coefficients and the source term satisfy 

\noindent {\bf Hypothesis A} \label{hypo} {\it Let $f \in
  L^2(\Omega)$, $0<\eps <1$ be a fixed arbitrary parameter and
  $\overset{\circ}{\Gamma_D} \neq \varnothing$.  The diffusion
  coefficients $A_{||} \in L^{\infty}(\Omega)$ and $A_{\perp} \in
  \MM_{d \times d} (L^{\infty} (\Omega))$ are supposed to satisfy
  \begin{gather}
        0<A_0 \le A_{||}(x) \le A_1\,, \quad  \textrm{f.a.a}\,\,\,x \in \Omega,
        \label{hyp1}
        \\[3mm]
        A_0 ||v||^2 \le v^t A_{\perp}(x) v \le A_1 ||v||^2\,,
        \quad \forall v\in \mathbb R^d\,\,\, \text{with} \,\,\, v\cdot
        b(x)=0\,\,\, \text{and} \,\,\,  \textrm{f.a.a}\,\,\, x \in \Omega,
        \label{hyp2}
  \end{gather}
with some constants $0<A_0 \le A_1$.
}\\
For the beginning we shall consider a constant anisotropy intensity
$\eps$ (no space dependence) in order to better understand and study
the construction of the AP-reformulation. Later on, in section
\ref{GENE} we shall generalize the AP-reformulation to variable $\eps$
cases.

As we intend to use the finite element method for the numerical
solution of the P-problem, let us put (\ref{P}) under variational
form. For this, let us introduce the Hilbert-space
\begin{equation} \label{V}
{\cal V}=\{v\in H^{1}(\Omega )\,,\text{ such that }v|_{\Gamma _{D}}=0\}\,, \quad (u,v)_{\cal V}:= (\nabla _{{||}}u, \nabla _{{||}}v) + (\nabla _{\perp }u, \nabla _{\perp }v)\,,
\end{equation}%
and its subspace
\begin{equation} \label{G}
{\cal G}=\{v\in {\cal V}\,,\text{ such that } \nabla_{||}v=0\}\,, \quad  (u,v)_{\cal G}:=  (\nabla _{\perp }u, \nabla _{\perp }v)\,,
\end{equation}%
where $(\cdot,\cdot)$ shall stand in the following for the standard
$L^2$-scalar product.\\
We are seeking thus for $u^\eps \in {\cal V}$, solution of
\be \label{Pvar}
(P)\,\,\,\quad
a_{||}(u^\eps,v) + \eps a_{\perp}(u^\eps,v)=\eps (f,v)\,, \quad
\forall v \in {\cal V}\,,
\ee
where  the continuous bilinear forms $a_{||} : \cal{V} \times \cal{V} \rightarrow
\RR$ and $a_{\perp}: \cal{V} \times \cal{V} \rightarrow \RR$ are given by
\be \label{bil}
\begin{array}{lll}
  \ds a_{||}(u,v)&:=&\ds \int_{\Omega} A_{||} \nabla_{||}
  u \cdot \nabla_{||}v\, dx\,, \quad  a_{\perp}(u,v):=\ds \int_{\Omega} ( A_{\perp} \nabla_{\perp}
  u) \cdot \nabla_{\perp}v\, dx\,.
\end{array}
\ee

Thanks to Hypothesis A and to the Lax-Milgram theorem, problem
(\ref{Pvar}) admits a unique solution $u^\eps \in {\cal V}$ for all
fixed $\eps >0$.  The parameter $0<\eps<1$ is responsible for the high
anisotropy of the problem, and its smallness induces severe numerical
difficulties. Indeed, putting formally $\eps=0$ in (\ref{Pvar}),
yields $a_{||}(u^0,v)=0$ for all $v \in {\cal V}$, which has
infinitely many solutions, constant along the field lines. We see thus
that the system (\ref{Pvar}) becomes degenerate in the limit $\eps
\rightarrow 0$. It means that solving (\ref{Pvar}) for $0 < \eps <<1$
with standard numerical schemes is very inadequate, since we have to
deal with very ill-conditioned linear systems.

However, as detailed in \cite{DDLNN}, the solution $u^\eps \in {\cal
  V}$ of (\ref{P}) is shown to tend as $\eps \rightarrow 0$ towards a
function $u^0 \in {\cal G}$, constant along the field lines of $b$ and
solution of the Limit model \be \label{LL} (L)\,\,\,\quad
\int_{\Omega} ( A_{\perp} \nabla_{\perp} u^0) \cdot \nabla_{\perp}v\,
dx = \int_{\Omega} f v\, dx \,, \quad \forall v \in {\cal G}\,.  \ee
Again, the Lax-Milgram theorem permits to show the existence and
uniqueness of a solution $u^0 \in {\cal G}$ of this Limit problem
(\ref{LL}).

The aim of the present work, is to introduce an AP-scheme (Asymptotic
Preserving), which shall permit to find numerically the solution of
problem (\ref{P}) (or (\ref{Pvar}) equivalently) in an accurate
manner, independently on $\eps$ and with only moderate requirements
concerning computer resources.  This shall be based on the
introduction of an equivalent reformulation of the P-model, which
shall permit to pass continuously from the reformulated P-model
(\ref{Pvar}) to the L-model (\ref{LL}) as $\eps \rightarrow 0$. In
other words, this reformulation of the P-model will simply be a
regular perturbation of the L-model, avoiding thus the encountered
problems of ill-posedness, when passing to the limit directly in the
singularly perturbed problem (\ref{Pvar}). This procedure leads to
some considerable numerical advantages for $0<\eps<<1$.

\subsection{An AP- reformulation of the problem (MM-problem)}\label{sec:AP}

To introduce a reformulation of the P-model, which shall be well-posed
in the limit $\eps \rightarrow 0$, we need the following Hilbert space
\begin{equation} \label{GL}
{\cal L}=\{q\in L^{2}(\Omega )~/~\nabla _{{||}}q\in
L^{2}(\Omega )\text{ and }q|_{\Gamma _{in}}=0\}\,, \quad (q,w)_{{\cal
    L}}:=(\nabla _{{||}}q, \nabla _{{||}}w)\,, \quad \forall q,w \in
{\cal L}\,. 
\end{equation}
Consider the following problem, called in the sequel
Asymptotic-Preserving problem based on a micro-macro decompostion
(MM-problem): find $(u^\eps,q^\eps)\in {\cal V} \times {\cal L}$,
solution of
\begin{equation}
  (MM)\,\,\, 
  \left\{ 
    \begin{array}{ll}
      \ds \int_{\Omega }( A_{\perp }\nabla _{\perp }u^\eps)\cdot \nabla _{\perp
      }v\,dx+\int_{\Omega } A_{||} \nabla _{||}q^\eps\cdot \nabla _{||}v\,dx=\int_{\Omega }fv\,dx, & \quad
      \forall v\in {\cal V} \\[3mm]
      \ds \int_{\Omega }A_{||} \nabla {}_{||}u^\eps\cdot \nabla {}_{||}w\,dx-\varepsilon \int_{\Omega }
      A_{||} \nabla _{||}q^\eps\cdot \nabla _{||}w\,dx=0, & \quad \forall w\in {\cal L}\,.
    \end{array}%
  \right.  \label{Pa}
\end{equation}
System (\ref{Pa}) is an equivalent reformulation (for fixed $\eps>0$)
of the original P-problem (\ref{Pvar}). Indeed, if $u^\eps\in\mathcal
V$ solves (\ref{Pvar}), then we can construct $q^\eps \in {\cal L}$
such that $\nabla_{||} q^\eps = (1/\eps) \nabla_{||} u^\eps$,
cf. Lemma \ref{lem1} below.  Indeed, for this observe that
$$
q \in {\cal L} \longmapsto \nabla_{||} q \in L^2(\Omega)\,,
$$
is a one to one mapping. This, in weak form, gives the second equation
of (\ref{Pa}).  Replacing then $\nabla_{||} u^\eps$ by $\eps
\nabla_{||} q^\eps$ inside (\ref{Pvar}), we see that $(u^\eps,q^\eps)$
solves also the first equation in (\ref{Pa}).  Conversely, if
(\ref{Pa}) has a solution $(u^\eps,q^\eps)\in {\cal V}\times {\cal L}$
then the second equation implies $\eps \nabla _{||}q^\eps=\nabla
_{||}u^\eps$, which inserted in the first one, leads to the weak
formulation (\ref{Pvar}).

This proves that the MM-formulation (\ref{Pa}) has a unique solution
$(u^\eps,q^\eps)\in {\cal V}\times {\cal L}$ for all $\eps>0$ and
$f\in L^2(\Omega)$, where $u^\eps \in {\cal V}$ is the unique solution
of the P-problem (\ref{Pvar}). The advantage of (\ref{Pa}) over
(\ref{Pvar}) consists in the fact that taking formally the limit $\eps
\to 0$ in (\ref{Pa}) leads to the correct limit problem (\ref{LL}).
Indeed, setting $\eps=0$ in the MM-formulation (\ref{Pa}), we obtain
the following problem (referred hereafter as the L-model): Find
$(u^0,q^0) \in {\cal V} \times {\cal L}$ such that
\begin{equation}
(L)\,\,\, \left\{ 
\begin{array}{ll}
\ds \int_{\Omega }( A_{\perp }\nabla _{\perp }u^0)\cdot \nabla _{\perp
}v\,dx+\int_{\Omega } A_{||} \nabla _{||}q^0\cdot \nabla _{||}v\,dx=\int_{\Omega }fv\,dx, & \quad
\forall v\in {\cal V} \\[3mm]
\ds \int_{\Omega }A_{||} \nabla {}_{||}u^0\cdot \nabla {}_{||}w\,dx=0, & \quad \forall w\in {\cal L}\,.
\end{array}%
\right.  \label{L}
\end{equation}
Remark that (\ref{L}) is formally an equivalent formulation of the
Limit problem (\ref{LL}).  In particular, if $(u^0,q^0) \in {\cal V}
\times {\cal L}$ is a solution of (\ref{L}), then $u^0 \in {\cal G}$,
where ${\cal G}$ is defined in (\ref{G}) and $u^0$ solves (\ref{LL}).
The additional unknown $q^0$ serves here as the Lagrange multiplier
responsible for the constraint $u^0 \in {\cal G}$. The existence of
this Lagrange multiplier $q^0 \in {\cal L}$ is not completely
straight-forward to prove, since we do not have an inf-sup property
for the bilinear form $a_{||}$ on the pair of spaces $\mathcal V
\times \mathcal L$.  Fortunately, we can prove the existence assuming
$f\in L^2(\Omega)$, cf. Theorem \ref{thm_cont}, thus establishing
rigorously the equivalence between (\ref{LL}) and (\ref{L}), at least
for $f\in L^2(\Omega)$. This shall be part of the aim of the next
subsection. The uniqueness is given by
\begin{lemma} \label{lem_unic} Suppose that Hypothesis A is satisfied,
  in particular that $f\in L^{2}(\Omega )$. Then the solution to
  (\ref{L}), if it exists, is unique.
\end{lemma}

\begin{proof}
  It is sufficient, due to linearity, to consider $f=0$. Let thus
  $(u^{0},q^{0})\in \mathcal{V}\times \mathcal{L}$ be the solution of
  (\ref{L}) for $f=0$. Taking then test functions $v \in {\cal G}$, we
  get immediately $u^{0}=0$, implying $%
  a_{||}(q^{0},v)=0$ for all $v\in \mathcal{V}$. By density arguments
  one gets then $q^{0}=0$.
\end{proof}

\begin{remark}
  In the previous paper \cite{DDLNN} we used the decomposition $u^\eps
  = p^\eps + q^\eps$ for the construction of the duality-based
  AP-scheme, where $p \in {\cal G}$ and $q \in {\cal A}$, with ${\cal
    A}$ the space of functions with zero average along the field
  lines. In the present paper we have a slightly different
  decomposition $u^\eps = p^\eps + \eps q^\eps$ hidden in
  (\ref{Pvar}). The space $\cal A$ was replaced by the space $\cal L$
  of functions vanishing on the inflow boundary.  In the previous
  version, $q^{\eps}$ was useless in the limit $\eps \to 0$, now it is
  a meaningfull Lagrange multiplier in this limit. Indeed, we prove in
  Theorem \ref{thm_cont} below that $q^\eps \to q^0$ as $\eps \to
  0$. This is the main reason why our new Micro-Macro AP-reformulation
  is more economical than the old Duality-based AP-reformulation.
\end{remark}

\subsection{The behaviour of the MM-problem as $\eps \rightarrow 0$} \label{SEC23}
The objective of this subsection is to study the behaviour of the
MM-solution $(u^\eps,q^\eps) \in {\cal V} \times {\cal L}$ of system
(\ref{Pa}) in the limit $\eps \rightarrow 0$, in particular to show
that it tends (in some sense) towards $(u^0,q^0) \in {\cal V} \times
{\cal L}$, solution of the Limit model (\ref{L}).  In order to prove
rigorously these results, we will suppose the following hypothesis on
the domain that tells us essentially that $\Omega $ is a tube composed
of field
lines of $b$.\\

\noindent {\bf Hypothesis B} {\it There exists a smooth coordinate
  system $(\xi _{1},\ldots ,\xi _{n})$ on $\Omega $ with $(\xi
  _{1},\ldots ,\xi _{n-1})\in D$, $\xi _{n}\in (0,1)$, $D$ being a
  smooth domain in $\mathbb{R}^{n-1}$, such that the field lines of
  $b$ are given by the coordinate lines $(\xi _{1},\ldots ,\xi
  _{n-1})=const$. Moreover, $\Gamma _{in}$ is represented by $\xi
  _{n}=0$, $%
  (\xi _{1},\ldots ,\xi _{n-1})\in D$; $\Gamma _{out}$ is represented
  by $\xi _{n}=1$, $(\xi _{1},\ldots ,\xi _{n-1})\in D$ and $\Gamma
  _{D}$ is represented by $\xi _{n}\in (0,1)$, $(\xi _{1},\ldots ,\xi
  _{n-1})\in
\partial D$.}

We proved in \cite{DDLNN} that such a coordinate system exists
provided \ the components of the boundaries $\Gamma _{in}$ and $\Gamma
_{out}$ admit smooth parametrizations and that $b$ penetrates $\Gamma
_{in}$ and $\Gamma _{out}$ at an angle that stays away from 0$.$

In the following we suppose that Hypotheses A and B hold true. We are then
able to prove the following 
\begin{theorem} \label{thm_cont} Let Hypothesis A and B be statisfied
  and moreover suppose that $A_{\perp }\in \mathbb{M}_{d\times
    d}(W^{2,\infty }(\Omega ))$ and $A_{||}\in W^{2,\infty }(\Omega
  )$. Then the MM-problem (\ref{Pa}) admits a unique solution
  $(u^{\eps},q^{%
    \eps})\in \mathcal{V}\times \mathcal{L}$ for any $\eps>0$, where
  $u^\eps$ is the unique solution of problem (\ref{Pvar}). There
  exists also a unique solution $(u^{0},q^{0})\in \mathcal{V}\times
  \mathcal{L}$ of the L-problem (\ref{L}), where $u^{0}\in
  \mathcal{G}$ solves problem (\ref{LL}). Moreover, we have the
  following convergences as $\varepsilon \rightarrow 0$%
  \begin{equation*}
    u^{\eps}\rightarrow u^{0}\ \quad \text{in}\quad \mathcal{V}\,,\quad \quad q^{%
      \eps}\rightharpoonup q^{0}\quad \text{in}\quad \mathcal{L}\,,
  \end{equation*}%
  and the following bounds hold  
  \begin{equation*}
    ||\nabla _{\perp }u^{\eps%
    }-\nabla _{\perp }u^{0}||_{L^{2}}\leq
    C\sqrt{\eps}||f||_{L^{2}}\,,\quad ||\nabla
    _{||}u^{\eps}||_{L^{2}}\leq C\eps\,||f||_{L^{2}} \,\,\,
    \text{and}\,\,\,\,\,||\nabla _{||}q^{%
      \eps}||_{L^{2}}\leq C||f||_{L^{2}}\,.
  \end{equation*}%
  with a constant $C>0$ independent of $\eps$ and $f$.
\end{theorem}

The proof of this theorem will use several lemmas which we prove first.

\begin{lemma}\label{lem1}
  For any $u\in H^{1}(\Omega )$ and $\varepsilon >0$, there exists a
  unique $%
  q\in \mathcal{L}$ satisfying $\varepsilon \nabla _{||}q=\nabla
  _{||}u$ a.e.  Moreover, if $u\in H^{2}(\Omega )$ then $q\in
  H^{1}(\Omega )$ and if $u\in \mathcal{V}\cap H^{2}(\Omega )$ then
  $q\in \mathcal{V}$.
\end{lemma}

\begin{proof}
  The idea of the proof is simply to say that $q$ is constructed by
  subtracting from $u$ its value on the inflow $\Gamma _{in}$ on each
  field line and dividing the result by $\varepsilon $. More formally
  speaking, let us introduce the operator $J$ that takes any function
  $\phi $ on $\Gamma _{in}$ and returns the function $p=J\phi $, which
  coincides with $\phi $ on $%
  \Gamma _{in}$ and is constant on the field lines. In the coordinate
  system introduced in Hypothesis B, the definition of the operator
  $J$ resumes to the following formula%
  \begin{equation*}
    (J\phi )(\xi _{1},\ldots ,\xi _{n-1},\xi _{n})=\phi (\xi _{1},\ldots ,\xi
    _{n-1}).
  \end{equation*}%
  This implies (going from the coordinates $\xi $ to the original
  ones) that $J $ maps $L^{2}(\Gamma _{in})$ to $L^{2}(\Omega )$,
  $H^{1}(\Gamma _{in})$ to $ H^{1}(\Omega )$ and $H_{0}^{1}(\Gamma
  _{in})$ to $\mathcal{V}$. Let us now define, for given $u$, the
  function $q=(u-J(u|_{\Gamma _{in}}))/\varepsilon $. Consequently, if
  $u\in H^{1}(\Omega )$ then $u|_{\Gamma _{in}}\in L^{2}(\Gamma
  _{in})$ and $q\in L^{2}(\Omega )$. If $u\in H^{2}(\Omega )$ then
  $u|_{\Gamma _{in}}\in H^{1}(\Gamma _{in})$ and $q\in H^{1}(\Omega
  )$. If, moreover, $u$ vanishes on $\Gamma _{D}$ then $u|_{\Gamma
    _{in}}\in H_{0}^{1}(\Gamma _{in})$ and $%
  q\in \mathcal{V}$.
\end{proof}

\begin{lemma}\label{lem2}
  Let Hypothesis A and B be statisfied and moreover suppose that
  $A_{\perp }\in \mathbb{M}_{d\times d}(W^{2,\infty }(\Omega ))$ and
  $A_{||}\in W^{2,\infty }(\Omega )$. Then the solution $u^{0}\in
  \mathcal{G}$ of (%
  \ref{LL}) belongs to $H^{2}(\Omega )$ and satisfies the estimates
  \begin{equation}
    ||u^{0}||_{H^{2}}\leq C||f||_{L^{2}}\,,  \label{estu0}
  \end{equation}%
  with a constant $C$ independent of $f$.
\end{lemma}

\begin{proof}
  Since $u^{0}$ is constant along the field lines, it is represented
  by $%
  u^{0}(\xi _{1},\ldots ,\xi _{n-1})$ in the notations of Hypothesis
  B. The same is true for the test functions $v$ in
  (\ref{LL}). Rewriting (\ref{LL}) in the $\xi $-coordinates gives
  \begin{equation*}
    \sum_{1\leq i,j\leq n}\sum_{1\leq k,l\leq n-1}\int_{D\times (0,1)}A_{\perp
      ,ij}\frac{\partial \xi _{k}}{\partial x_{i}}\frac{\partial \xi _{l}}{%
      \partial x_{j}}\frac{\partial u^{0}}{\partial \xi _{k}}\frac{\partial v}{%
      \partial \xi _{l}}\left\vert \frac{\partial x}{\partial \xi }\right\vert
    d\xi =\int_{D\times (0,1)}fv\left\vert \frac{\partial x}{\partial \xi }%
    \right\vert d\xi 
  \end{equation*}%
  Introducing the positive definite matrix $C_{kl}=\sum_{1\leq i,j\leq
    n}A_{\perp ,ij}\frac{\partial \xi _{k}}{\partial
    x_{i}}\frac{\partial \xi _{l}}{\partial x_{j}}\left\vert
    \frac{\partial x}{\partial \xi }\right\vert $ and integrating on
  $\xi _{n}$ over $(0,1)$ yields%
  \begin{equation*}
    \sum_{1\leq k,l\leq n-1}\int_{D}C_{kl}\frac{\partial u^{0}}{\partial \xi _{k}%
    }\frac{\partial v}{\partial \xi _{l}}d\xi _{1}\cdots d\xi
    _{n-1}=\int_{D\times (0,1)}fv\left\vert \frac{\partial x}{\partial \xi }%
    \right\vert d\xi .
  \end{equation*}%
  We see that it is the weak formulation of an elliptic equation for
  $u^{0}$ on $D$ with the boundary conditions $u^{0}=0$ on $\partial
  D$. By regularity results for elliptic equations, we see immediately
  that $u^{0}\in H^{2}(D)$ if $f\in L^{2}(\Omega )$ with continuous
  dependence of $u^{0}$ on $f$.  Reminding that $u^{0}$ does not
  depend on $\xi _{n}$ and going back to the original coordinates
  yields (\ref{estu0}).
\end{proof}

\medskip 
\begin{proof}[Proof of Theorem \ref{thm_cont}]
  The existence and uniqueness of a solution of the MM-problem
  (\ref{Pa}) was shown in the previous section. Note that
  $u^{\varepsilon }$ is in fact in $H^{2}(\Omega )$ by the regularity
  results for elliptic equations since $u^{\varepsilon }$ solves
  (\ref{P}). Moreover, $q^{\varepsilon }\in \mathcal{V}$ thanks to
  Lemma \ref{lem1}.

  Taking in the first equation of (\ref%
  {Pa}) test functions $v \in \mathcal{G}$ , we deduce%
  \begin{equation*}
    a_{\perp }(u^{\eps},v)=(f,v)=a_{\perp }(u^{0},v),\quad \forall v\in \mathcal{%
      G}\,.
  \end{equation*}%
  Choose now $v:=u^{\eps}-\varepsilon q^{\varepsilon }-u^{0}$ where
  $u^{0}\in \mathcal{G}$ is the unique solution of (\ref{LL}). Thanks
  to the second eqaution in (\ref{Pa}), $\nabla _{||}v=0$ a.e. in
  $\Omega $ so that $v\in \mathcal{G}$. Substituting it in the
  equation above gives%
  \begin{equation*}
    a_{\perp }(u^{\eps},u^{\eps}-\varepsilon q^{\varepsilon }-u^{0})=a_{\perp
    }(u^{0},u^{\eps}-\varepsilon q^{\varepsilon }-u^{0})\,,
  \end{equation*}%
  implying
  \begin{equation}
    a_{\perp }(u^{\eps}-u^{0},u^{\eps}-u^{0})-\varepsilon a_{\perp }(u^{\eps%
    },q^{\varepsilon })=-\varepsilon a_{\perp }(u^{0},q^{\varepsilon }).
    \label{aga1}
  \end{equation}%

  Take now $v:=\varepsilon q^{\varepsilon }$ in the first equation of
  (\ref{Pa}) and add it to (\ref{aga1}). This yields
  \begin{equation}
    a_{\perp }(u^{\eps}-u^{0},u^{\eps}-u^{0})+\varepsilon a_{||}(q^{\varepsilon
    },q^{\varepsilon })=\varepsilon (f,q^{\varepsilon })-\varepsilon a_{\perp
    }(u^{0},q^{\varepsilon }).  \label{aga2}
  \end{equation}
  Since $u^{0}\in H^{2}(\Omega )$ by Lemma \ref{lem2} we can integrate
  by parts in $a_{\perp }(u^{0},q^{\eps})$:
  \begin{align*}
    -a_{\perp }(u^{0},q^{\eps})& =-\int_{\Omega } \left(A_{\perp
      }\nabla _{\perp }u^{0} \right)\cdot \nabla _{\perp
    }q^{\eps}dx=-\int_{\Omega }\left[ (Id-b\otimes b)A_{\perp
      }\nabla u^{0} \right]\cdot \nabla q^{\eps}dx \\
    & =-\int_{\Gamma _{out}}\left[ (Id-b\otimes b)A_{\perp }\nabla
      u^{0}\right] \cdot nq^{\eps%
    }d\sigma +\int_{\Omega }\left[\nabla \cdot ((Id-b \otimes
      b)A_{\perp }\nabla u^{0}) \right] q^{%
      \eps}dx \\
    & \leq C||u^{0}||_{H^{2}(\Omega)}\left( ||q^{\eps}||_{L^{2}(\Gamma
        _{out})}+||q^{\eps%
      }||_{L^{2}(\Omega )}\right)
  \end{align*}%
  since $\nabla u^{0}$ has a trace on $\Gamma _{out}$ and its norm in
  $L^{2}(\Gamma _{out})$ is bounded by
  $C||u^{0}||_{H^{2}(\Omega)}$. Thus, (\ref{aga2}) tells us
  \begin{equation*}
    ||\nabla _{||}q^{\eps}||_{L^{2}(\Omega)}^{2}\leq Ca_{||}(q^{\eps},q^{\eps})\leq
    C||f||_{L^{2}}||q^{\eps}||_{L^{2}(\Omega )}+C||u^{0}||_{H^{2}(\Omega)}\left( ||q^{%
        \eps}||_{L^{2}(\Gamma _{N})}+||q^{\eps}||_{L^{2}(\Omega )}\right) .
  \end{equation*}%
  By the Poincar\'{e} and trace inequalities (proved easily by passing
  to the $\xi $-coordinates of Hypothesis B) and by Lemma \ref{lem2},
  we have
  \begin{equation*}
    ||q^{\eps}||_{L^{2}(\Omega )}\leq C||\nabla _{||}q^{\eps}||_{L^{2}(\Omega)},\quad
    ||q^{\eps}||_{L^{2}(\Gamma _{out})}\leq C||\nabla _{||}q^{\eps}||_{L^{2}(\Omega)},%
    \text{ and }||u^{0}||_{H^{2}(\Omega)}\leq C||f||_{L^{2}(\Omega)}
  \end{equation*}%
  so that 
  \begin{equation*}
    ||\nabla _{||}q^{\eps}||_{L^{2}(\Omega)}\leq C||f||_{L^{2}(\Omega)}.
  \end{equation*}%
  This gives immediately also the estimate $||\nabla _{||}u^{\eps
  }||_{L^{2}}\leq C\varepsilon ||f||_{L^{2}}$ since $\nabla
  _{||}u^{\eps }=\varepsilon \nabla _{||}q^{\eps}$. Returning to
  (\ref{aga2}), we observe that
  \begin{equation*}
    ||\nabla _{\perp }u^{\eps}-\nabla _{\perp }u^{0}||_{L^{2}}^{2}\leq
    \varepsilon
    (f,q^{\varepsilon })-\varepsilon a_{\perp }(u^{0},q^{\varepsilon })\leq
    C\varepsilon ||f||_{L^{2}(\Omega)}^{2}\,,
  \end{equation*}%
  so that all the estimates are proved. They imply immediately the
  strong convergence $u^{\varepsilon }\rightarrow u^{0}$ in
  $H^{1}(\Omega )$.

  It remains to prove the weak convergence of $q^{\varepsilon }$ and
  existence of $q^{0}$ that solves (\ref{L}). Since the familly of
  functions $q^{\varepsilon }$ is bounded in $\mathcal{L}$, there is a
  subsequence $q^{\varepsilon _{n}}$ weakly converging to some
  $q^{0}\in \mathcal{L}$ as $\varepsilon _{n}\rightarrow 0$. Taking
  the limit $\varepsilon _{n}\rightarrow 0$ in (\ref{Pa}), we see that
  $(u^{0},q^{0})$ solves (\ref{L}). We know already that the solution
  $(u^{0},q^{0})$ is unique, cf. Lemma \ref{lem_unic}. It means that
  any converging sequence $q^{\varepsilon _{n}}$ (with $\varepsilon
  _{n}\rightarrow 0$) has $q^{0}$ as its limit. Hence, $q^{\varepsilon
  }\rightharpoonup q^{0}$ as $\varepsilon \rightarrow 0$. This
  finishes the proof.
\end{proof}

The next theorem shows some nice $H^2$-regularity results for the
unique solution $u^{ \eps} \in {\cal V}$ of the P-problem
(\ref{Pvar}). This result is proven for the moment only for a
simplified geometry: $\Omega:=(0,L_x) \times (0,L_y)$ and $b=(0,1)$
assumed constant and aligned in the $y$ direction.  Let us thus study
the system
\begin{equation}\label{Psim}
  \left\{
    \begin{array}{ll}
      -{1 \over \eps} \partial_y (A_y \partial_y u^\eps ) - \partial_x (A_x \partial_x u^\eps ) =f \,, & \text{in}\,\,\, \Omega\\[2mm]
      \partial_y u^\eps =0 \,, & \text{for} \,\,\, y=0,L_y\\[2mm]
      u^\eps =0\,, & \text{for}\,\,\, x=0,L_x\,.
    \end{array}
  \right.
\end{equation}

\begin{theorem}\label{der2}
  Take $\Omega:=(0,L_x) \times (0,L_y)$, $b=(0,1)$, suppose that
  Hypothesis A is satisfied and moreover that $A_x=(A_{\perp
  })_{11}\in W^{2,\infty }(\Omega )$ and $A_y=A_{||}\in W^{2,\infty
  }(\Omega )$.  Then $u^{\eps}$, the unique solution of (\ref{Psim}),
  belongs to $H^{2}(\Omega )$ and we have the estimates
  \begin{equation}
    ||\partial_x u^\eps||_{L^{2}}^{2}+\frac{1}{\varepsilon ^{2}}%
    ||\partial_y u^\eps||_{L^{2}}^{2}\leq C||f||_{L^{2}}^{2}\,,
    \label{est_H1}
  \end{equation}%
  \begin{equation}
    ||\partial_{xx}u^\eps||_{L^{2}}^{2}+\frac{1}{\varepsilon }%
    ||\partial_{xy}u^\eps||_{L^{2}}^{2}+\frac{1}{\varepsilon ^{2}}%
    ||\partial_{yy}u^\eps||_{L^{2}}^{2}\leq C||f||_{L^{2}}^{2}\,,
    \label{est2d}
  \end{equation}
  with $C>0$ a constant independent of $\eps$ and $f$.
\end{theorem}
\begin{remark}
  The estimate (\ref{est_H1}) is already proven in Theorem
  \ref{thm_cont} in a more general context. However, we provide below
  an alternative proof for it, which consists in the interplay with
  the estimates for the second derivatives and which does not require
  the Lemmas proven above. This alternative proof is thus simpler than
  that of Theorem \ref{thm_cont} presented above, but it is not
  straight forward to generalize it to the case of an arbitrary
  geometry of $\Omega$ and an arbitrary field $b$.
\end{remark}
\begin{proof}[Proof of Theorem \ref{der2}]
  Standard elliptic
  results permit to show, under the additional hypothesis of this
  theorem, that ${u^\eps} \in H^2(\Omega)$.

  First remark that multiplying the equation by ${u^\eps}$, integrating over
  $\Omega$ yields immediately by integration by parts the
  $H^1$-estimate
  \begin{equation} \label{H1}
    {1 \over \eps} ||\nabla_y {u^\eps}||^2_{L^{2}} + ||\nabla_x {u^\eps}||^2_{L^{2}} \le C ||f||^2_{L^{2}}\,.
  \end{equation}
  Rewriting now the equation as 
  $$
  -{1 \over \eps} A_y \partial_{yy} {u^\eps} - A_x \partial_{xx} {u^\eps} =f+{1 \over \eps} (\partial_y A_y) \partial_y {u^\eps} + (\partial_x A_x) \partial_x {u^\eps} \,,
  $$
  multiplying it by $-\partial_{yy} {u^\eps} -\partial_{xx} {u^\eps}$ and
  integrating over $\Omega$ yields (by integration by parts)
  $$
  \begin{array}{l}
    {1 \over \eps} || \sqrt{A_y} \partial_{yy} {u^\eps}||^2 +{1 \over \eps} || \sqrt{A_y} \partial_{xy} {u^\eps}||^2 + || \sqrt{A_x} \partial_{xy} {u^\eps}||^2 + || \sqrt{A_x} \partial_{xx} {u^\eps}||^2  \\[3mm]
    \hspace{1cm} \le C \left[ ||f|| \,(||\partial_{xx} {u^\eps}|| +
      ||\partial_{yy} {u^\eps}||) +  { 1 \over \eps} ( || \partial_y {u^\eps} ||\,
      || \partial_{xy} {u^\eps}||+  || \partial_y {u^\eps} ||\,  || \partial_{yy}
      {u^\eps}||) +  || \partial_x {u^\eps} ||\,  ( || \partial_{xy} {u^\eps}||+   || \partial_{xx} {u^\eps}||) \right]
  \end{array} 
  $$
  Using now the $H^1$-estimate (\ref{H1}), in particular that
  $|| \partial_y {u^\eps} || \le C \sqrt{\eps} ||f||$ one gets
  $$
  \begin{array}{l}
    {1 \over \eps} || \sqrt{A_y} \partial_{yy} {u^\eps}||^2 +{1 \over \eps} || \sqrt{A_y} \partial_{xy} {u^\eps}||^2 + || \sqrt{A_x} \partial_{xy} {u^\eps}||^2 + || \sqrt{A_x} \partial_{xx} {u^\eps}||^2  \\[3mm]
    \hspace{1cm} \le C  ||f|| \left( ||\partial_{xx} {u^\eps}|| +  ||\partial_{yy} {u^\eps}|| +  { 1 \over \sqrt{\eps}}  || \partial_{xy} {u^\eps}||+  { 1 \over \sqrt{\eps}}   || \partial_{yy} {u^\eps}|| +   || \partial_{xy} {u^\eps}|| \right)\,,
  \end{array} 
  $$
  yielding immediately by Young inequality
  $$
  \begin{array}{l}
    {1 \over \eps} || \partial_{yy} {u^\eps}||^2 +{1 \over \eps} || \partial_{xy} {u^\eps}||^2 + || \partial_{xx} {u^\eps}||^2  \le C ||f||^2\,.
  \end{array} 
  $$
  Coming now back to the equation
  $$
  -{1 \over \eps} \partial_y (A_y \partial_y {u^\eps})  =f + \partial_x (A_x \partial_x {u^\eps})\,,
  $$
  one gets with the last estimates
  $$
  {1 \over \eps} ||\partial_y (A_y \partial_y {u^\eps})|| \le C ||f||\,.
  $$
  Poincar\'e's inequality permits then to estimate
  $$
  ||A_y \partial_y {u^\eps} || \le C ||\partial_y (A_y \partial_y {u^\eps})|| \le C \eps ||f||\,,
  $$
  yielding $|| \partial_y {u^\eps}|| \le  C \eps ||f||$ and thus (\ref{est_H1}). Coming again back to the equation
  $$
  -{1 \over \eps} A_y \partial_{yy} {u^\eps}  =f + \partial_x (A_x \partial_x {u^\eps}) + {1 \over \eps} (\partial_y A_y) \partial_y {u^\eps}\,,
  $$ 
  permits to show (\ref{est2d}) and to conclude the proof.
\end{proof}

\section{Numerical results} \label{SEC4}
This section is devoted to the numerical simulation of the anisotropic
P-problem (\ref{P}) via the here introduced MM-scheme and to the
numerical illustration of its advantages.
\subsection{A finite element discretization} \label{SEC31}
Let us introduce a discretization of the domain $\Omega $ of size $h$
and a finite element space ${\cal V}_{h}$ of type $\mathbb P_{k}$ or
$\mathbb Q_{k}$ on this mesh. We assume that the boundary conditions
on $\Gamma _{D}$ are enforced in the definition of ${\cal V}_{h}$,
i.e. ${\cal V}_{h}\subset {\cal V}$. Consider then the subspace ${\cal
  L}_h$ of ${\cal V}_{h}$ defined by ${\cal L}_{h}$ $={\cal V}_{h}\cap
{\cal L}$
\begin{equation*}
  {\cal L}_{h}=\{q_{h}\in {\cal V}_{h}\,,\text{ such that }q_{h}|_{\Gamma _{in}}=0\}.
\end{equation*}
This choice is explained in more details in our previous paper
\cite{DDLNN}.  The standard discretization of (\ref{Pa}) writes then:
find $(u_{h}^\eps,q_{h}^\eps) \in {\cal V}_{h} \times {\cal L}_h$ such
that
\begin{equation}
  (MM)_h\,\,\, 
  \left\{ 
    \begin{array}{l}
      a_{\perp }(u_{h}^\eps,v_{h})+a_{{||}}(q_{h}^\eps,v_{h})=\int_{\Omega }fv_{h}\, dx,\quad
      \forall v_{h}\in {\cal V}_{h} \\ [3mm]
      a_{{||}}(u_{h}^\eps,w_{h})-\eps a_{{|| }}(q_{h}^\eps,w_{h})=0,\quad \forall
      w_{h}\in {\cal L}_h\,.%
    \end{array}%
  \right.  
  \label{Ph}
\end{equation}%

This section concerns the detailed study of the obtained numerical
results. In particular, we compare the method presented herein with
the duality-based AP-approach developed in our previous article
\cite{DDLNN} and present the convergence of the new scheme for an
arbitrary anisotropy field $b$ and constant $\varepsilon$ test case,
the convergence being uniform in $\eps$. The detailed numerical
analysis shall be presented in a forthcoming work.

\subsection{Discretization} \label{Discr}
Let us present, for simplicity, the discretization in a 2D case, the
3D case being a simple extension. The here considered computational
domain $\Omega $ is a square $\Omega = [0,1]\times [0,1]$. All
simulations are performed on structured meshes. Let us introduce the
Cartesian, homogeneous grid
\begin{gather}
  x_i = i / N_x \;\; , \;\; 0 \leq i \leq N_x \,, \quad
  y_j = j / N_y \;\; , \;\; 0 \leq j \leq N_y
  \label{eq:Jp8a},
\end{gather}
where $N_x$ and $N_y$ are positive even constants, corresponding to
the number of discretization intervals in the $x$-
resp. $y$-direction. The corresponding mesh-sizes are denoted by $h_x
>0$ resp. $h_y >0$. Choosing a $\mathbb Q_2$ finite element method
($\mathbb Q_2$-FEM), based on the following quadratic base functions

\begin{gather}
  \theta _{x_i}=
  \left\{
    \begin{array}{ll}
      \frac{(x-x_{i-2})(x-x_{i-1})}{2h_x^{2}} & x\in [x_{i-2},x_{i}],\\
      \frac{(x_{i+2}-x)(x_{i+1}-x)}{2h_x^{2}} & x\in [x_{i},x_{i+2}],\\
      0 & \text{else}
    \end{array}
  \right.\,, \quad 
  \theta _{y_j} =
  \left\{
    \begin{array}{ll}
      \frac{(y-y_{j-2})(y-y_{j-1})}{2h_y^{2}} & y\in [y_{j-2},y_{j}],\\
      \frac{(y_{j+2}-y)(y_{j+1}-y)}{2h_y^{2}} & y\in [y_{j},y_{j+2}],\\
      0 & \text{else}
    \end{array}
  \right.
  \label{eq:Js8a1}
\end{gather}
for even $i,j$  and
\begin{gather}
  \theta _{x_i}=
  \left\{
    \begin{array}{ll}
      \frac{(x_{i+1}-x)(x-x_{i-1})}{h_x^{2}} & x\in [x_{i-1},x_{i+1}],\\
      0 & \text{else}
    \end{array}
  \right.\,, \quad 
  \theta _{y_j} =
  \left\{
    \begin{array}{ll}
      \frac{(y_{j+1}-y)(y-y_{j-1})}{h_y^{2}} & y\in [y_{j-1},y_{j+1}],\\
      0 & \text{else}
    \end{array}
  \right.
  \label{eq:Js8a2}
\end{gather}
for odd $i,j$, we define
$$
W_h := \{ v_h = \sum_{i,j} v_{ij}\,  \theta_{x_i} (x)\,  \theta_{y_j}(y)\}\,.
$$
The spaces ${\cal V}_h$ and ${\cal L}_h$ are then defined by
\begin{equation*}
  {\cal V}_{h}=\{u_{h}\in {\cal W}_{h}\,,\text{ such that
  }u_{h}|_{\Gamma _{D}}=0\}\,, \quad {\cal L}_{h}=\{q_{h}\in {\cal V}_{h}\,,\text{ such that }q_{h}|_{\Gamma _{in}}=0\}.
\end{equation*}
The matrix elements are computed using the 2D Gauss quadrature
formula, with 3 points in the $x$ and $y$ direction:
\begin{gather}
  \int_{-1}^{1}\int_{-1}^{1}f (x,y) =
  \sum_{i,j=-1}^{1} \omega _{i}\omega _{j} f (x_i,y_j)\,,
  \label{eq:Jk9a}
\end{gather}
where $x_0=y_0=0$, $x_{\pm 1}=y_{\pm 1}=\pm\sqrt {\frac{3}{5}}$,
$\omega _0 = 8/9$ and $\omega _{\pm 1} = 5/9$, which is exact for
polynomials of degree 5.

\subsection{Duality-Based (DB) asymptotic-preserving method}
In order to compare the new, here introduced, MM-reformulation
(\ref{Pa}) with the previously considered \cite{DDLNN} duality-based
AP-scheme, let us briefly recall the former one. This alternative
approach is based on the following orthogonal decomposition of the
solution of the original problem (\ref{P}): $u = p + q$, where $p\in
\cal G$ and $q \in \cal A$. The vector space
\begin{gather}
  \mathcal G = \{ u \in \mathcal V \ | \ \nabla_\parallel u =0\}\,,
  \label{eq:Jd8a}
\end{gather}
is the Hilbert space of functions, which are constant along the field
lines of $b$ and was introduced in (\ref{G}). The vector space $\cal
A$ is the $L^2$-orthogonal complement to $\cal G$ in $\cal V$, defined
by
\begin{gather}
  \mathcal A : = 
   \{ u \in \mathcal V \,\,|\,\,\, (u,v ) =0 \;\; , \;\;  \forall
  v \in \mathcal G\}\,, \quad {\cal V}={\cal G} \oplus^{\perp} {\cal A}\,.
  \label{eq:Jg8a}
\end{gather}
Hence, the subspace $\cal A$ contains the functions that have zero
average along the field lines $b$.

This decomposition leads to the system: find $(p^\eps,q^\eps)\in \cal
G \times \cal A$ such that:
\begin{gather}
  \left\{
    \begin{array}{ll}
      \displaystyle
      a_{\perp} (p^{\varepsilon },\eta ) + a_{\perp} (q^{\varepsilon},\eta ) = (f,\eta )
      & \forall \eta \in \mathcal G\\[3mm]
      \displaystyle
      a_{||} (q^{\varepsilon },\xi ) + \varepsilon a_{\perp}
      (q^{\varepsilon},\xi) + \varepsilon a_{\perp} (p^{\varepsilon},\xi)
      = \varepsilon (f, \xi )
      & \forall \xi \in \mathcal A
    \end{array}
  \right.
  \label{eq:Ji8a},
\end{gather}
which is asymptotic-preserving, well-posed and well-conditioned
regardless of the value of $\varepsilon$. The discretization of the
vector spaces $\cal A$ and $\cal G$ is achieved by means of a Lagrange
multiplier technique: first we explore the orthogonality of $\cal A$
and $\cal G$ in order to avoid the direct discretization of the space
$\cal A$. The thus obtained system reads: find
$(p^\eps,q^\eps,l^\eps)\in \cal G \times \cal V \times \cal G$ such
that
\begin{gather}
  \left\{
    \begin{array}{ll}
      \displaystyle
      a_{\perp} (p^\eps,\eta ) + a_{\perp} (q^\eps,\eta )
      = (f,\eta )
      & \forall \eta \in \mathcal G, \\[3mm]
      \displaystyle
      a_{||} (q ^\eps,\xi ) + \varepsilon a_{\perp}
      (q ^\eps,\xi) + \varepsilon a_{\perp} (p ^\eps,\xi)
      + \left( l ^\eps , \xi \right)
      = \varepsilon (f, \xi ) 
      & \forall \xi \in \mathcal V , \\[3mm]
      \displaystyle
      \left( q ^\eps, \chi \right) =0
      & \forall \chi \in \mathcal G,
    \end{array}
  \right.
  \label{eq:Jj8a}
\end{gather}
where $l ^\eps $ is a Lagrange multiplier. The additional term $(l
^\eps,\xi)$ in the second equations allows us to replace the vector
space $\mathcal A$ by $\mathcal V$. The third equation forces $q
^\eps$ to belong to $\mathcal A$.

Afterwards, the definition of the space $\cal G$ is used to obtain a
system which does not require the direct discretization of $\cal
G$. The resulting system reads now: find $(p ^\eps,\;\lambda ^\eps,\;q
^\eps,\;l ^\eps,\;\mu ^\eps) \in \mathcal V\times \mathcal L\times
\mathcal V\times \mathcal V \times \mathcal L$ such that
\begin{gather}
  (DB)\,\,\,
  \left\{
    \begin{array}{l}
      \displaystyle
      a_{\perp} (p ^\eps , \eta ) +
      a_{\perp} (q ^\eps , \eta ) + a_{||}(\eta,\lambda ^\eps)
      = \left(f,\eta  \right) \,, \quad \forall \eta \in \mathcal V\,,
      \\[3mm]
      \displaystyle
      a_{||}( p ^\eps,\kappa)=0\,,\quad \forall \kappa \in \mathcal L\,, 
      \\[3mm]
      \displaystyle
      a_{||} (q ^\eps, \xi ) +
      \varepsilon a_{\perp} (q ^\eps , \xi ) +
      \varepsilon a_{\perp} (p ^\eps , \xi ) +
      \left( l ^\eps, \xi \right)
      = \varepsilon \left(f,\xi  \right)  \,, \quad \forall \xi \in \mathcal V\,,
      \\[3mm]
      \displaystyle
      \left( q ^\eps, \chi \right) +
      a_{||}(\chi, \mu ^\eps)=0\,, \quad \forall \chi \in \mathcal V\,,
      \\[3mm]
      \displaystyle
      a_{||}(l ^\eps,\tau)=0\,, \quad \forall \tau \in \mathcal  L\,,
    \end{array}
  \right.
  \label{DBAP}
\end{gather}
with $\cal L$ being a Lagrange multiplier space defined by
(\ref{GL}). Two additional Lagrange multipliers $\lambda^\eps $ and
$\mu ^\eps $ are introduced in order to replace $\mathcal G$ by the
bigger and easier to implement vector-space $\mathcal V$. For a more
detailed presentation of the duality-based asymptotic preserving
reformulation, we refer to \cite{DDLNN}.

This decomposition of the solution $u \in {\cal V}$ into two parts: a
mean part $p \in {\cal G}$ and the fluctuating part $q \in {\cal A}$
with zero average along the field lines, may seem more intuitive than
the new decomposition presented in this paper. This feature however
has its drawbacks. The fact that we had to introduce three additional
unknowns increases significantly the computational complexity of the
problem. In the following, we compare the DB Asymptotic-Preserving
approach (\ref{DBAP}) with the new MM Asymptotic-Preserving approach
(\ref{Pa}) and show that the new method is superior in terms of memory
requirements and computational time, while the accuracy remains the
same. In particular, we demonstrate that convergence is uniform in
$\eps$.

\subsection{Numerical tests}\label{sec:test case}
\subsubsection{2D test case, constant $\varepsilon$, uniform and aligned $b$-field}\label{sec:tc_bc}
In this section we compare the numerical results obtained via the
$\mathbb Q_2$-FEM described in Section \ref{Discr}, and applied to the
Singular Perturbation model (\ref{P}), the Duality-Based model
(\ref{DBAP}) and the Micro-Macro reformulation (\ref{Pa}). In all
numerical tests we set $A_\perp = Id$ and $A_\parallel = 1$. We start
with a simple test case, where the analytical solution is known. Let
the source term $f$ be given by
\begin{gather}
  f = \left(4 + \varepsilon  \right) \pi^{2} \cos \left( 2\pi x\right)
  \sin \left(\pi y \right) +
  \pi^{2} \sin \left(\pi y \right)
  \label{eq:J87a}
\end{gather}
and let the $b$ field be aligned with the $x$-axis. Hence, the solution
$u^{\varepsilon }$ of (\ref{P}) is given by
\begin{gather*}
  u^{\varepsilon } = \sin \left(\pi y \right) + \varepsilon \cos \left( 2\pi x\right)
  \sin \left(\pi y \right)
\end{gather*}

We denote by $u_P$, $u_D$ resp. $u_A$ the numerical solutions of the
Singular Perturbation model (\ref{P}), the Duality-Based Asymptotic
Preserving model (\ref{DBAP}) resp. the Micro-Macro Asymptotic Preserving
reformulation (\ref{Pa}). The comparison will be done in the
$L^{2}$-norm as well as the $H^{1}$-norm. The linear systems obtained
after discretization of the three methods are solved using the same
numerical algorithm --- LU decomposition implemented in a solver
MUMPS\cite{MUMPS}.

\def\xxxa{0.45\textwidth}
\begin{figure}[!ht] 
  \centering
  \subfigure[$L^{2}$ error for a grid with $50\times 50$ points.]
  {\includegraphics[angle=-90,width=\xxxa]{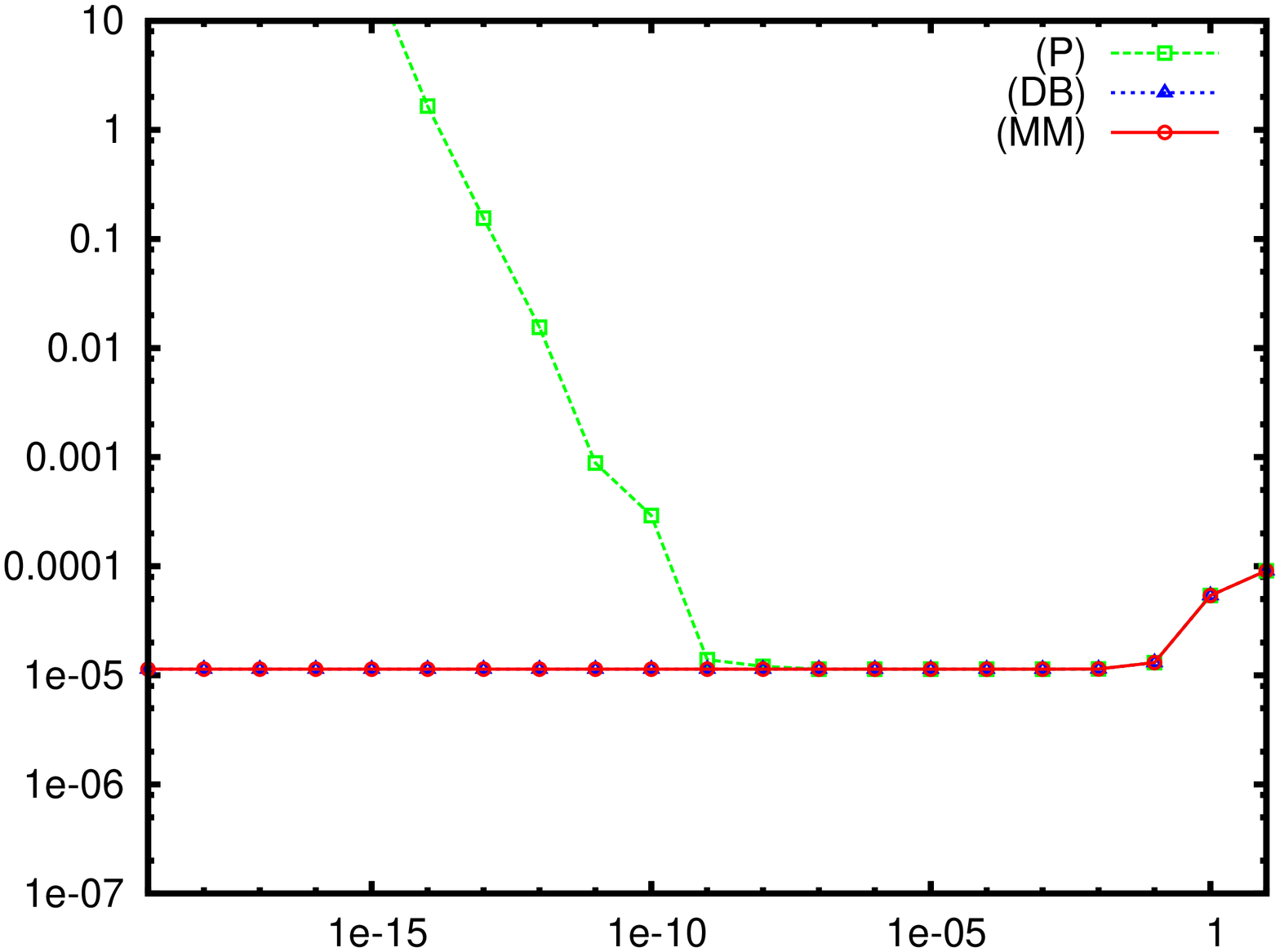}}
  \subfigure[$H^{1}$ error for a grid with $50\times 50$ points.]
  {\includegraphics[angle=-90,width=\xxxa]{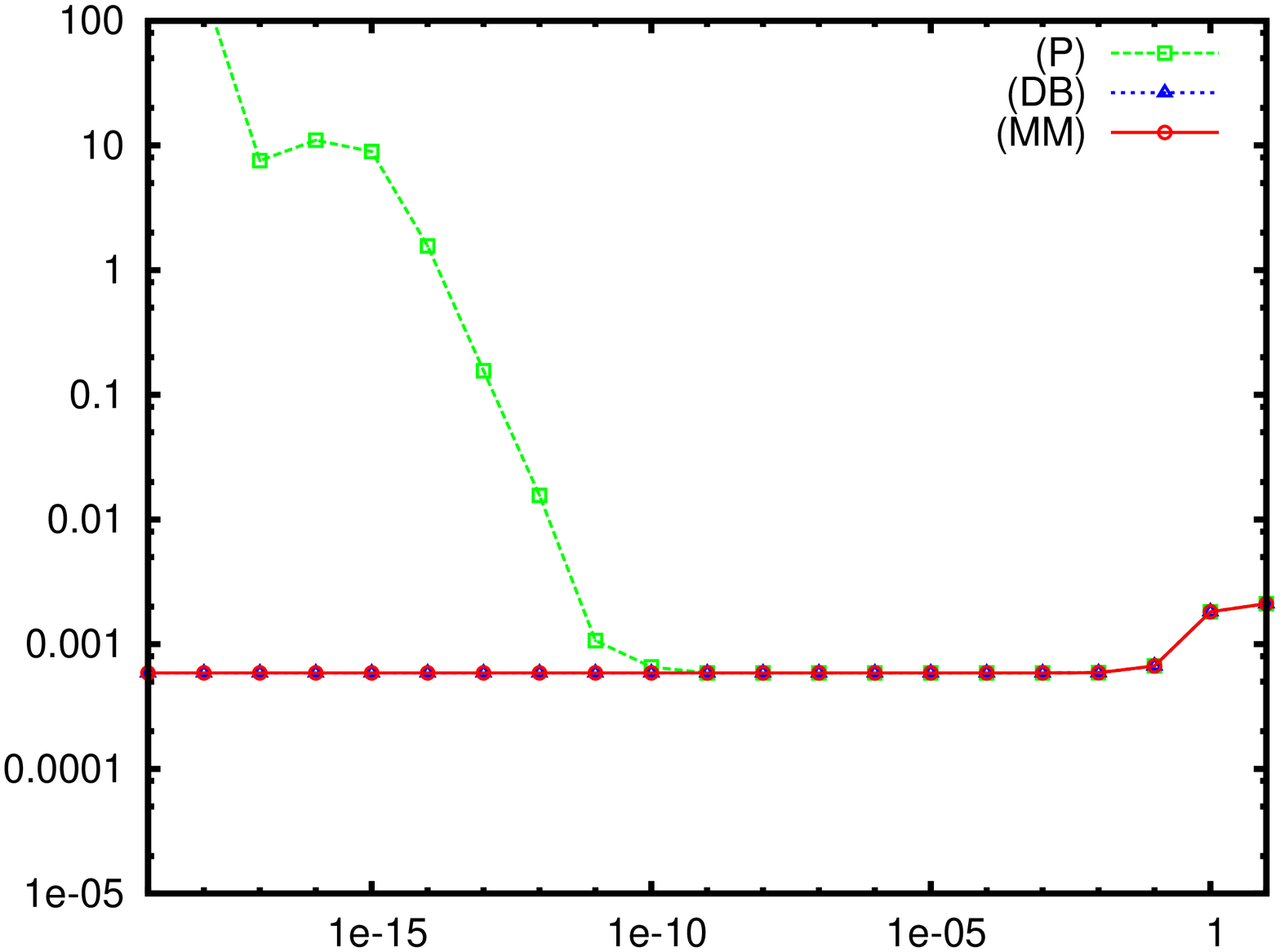}}

  \subfigure[$L^{2}$ error for a grid with $100\times 100$ points.]
  {\includegraphics[angle=-90,width=\xxxa]{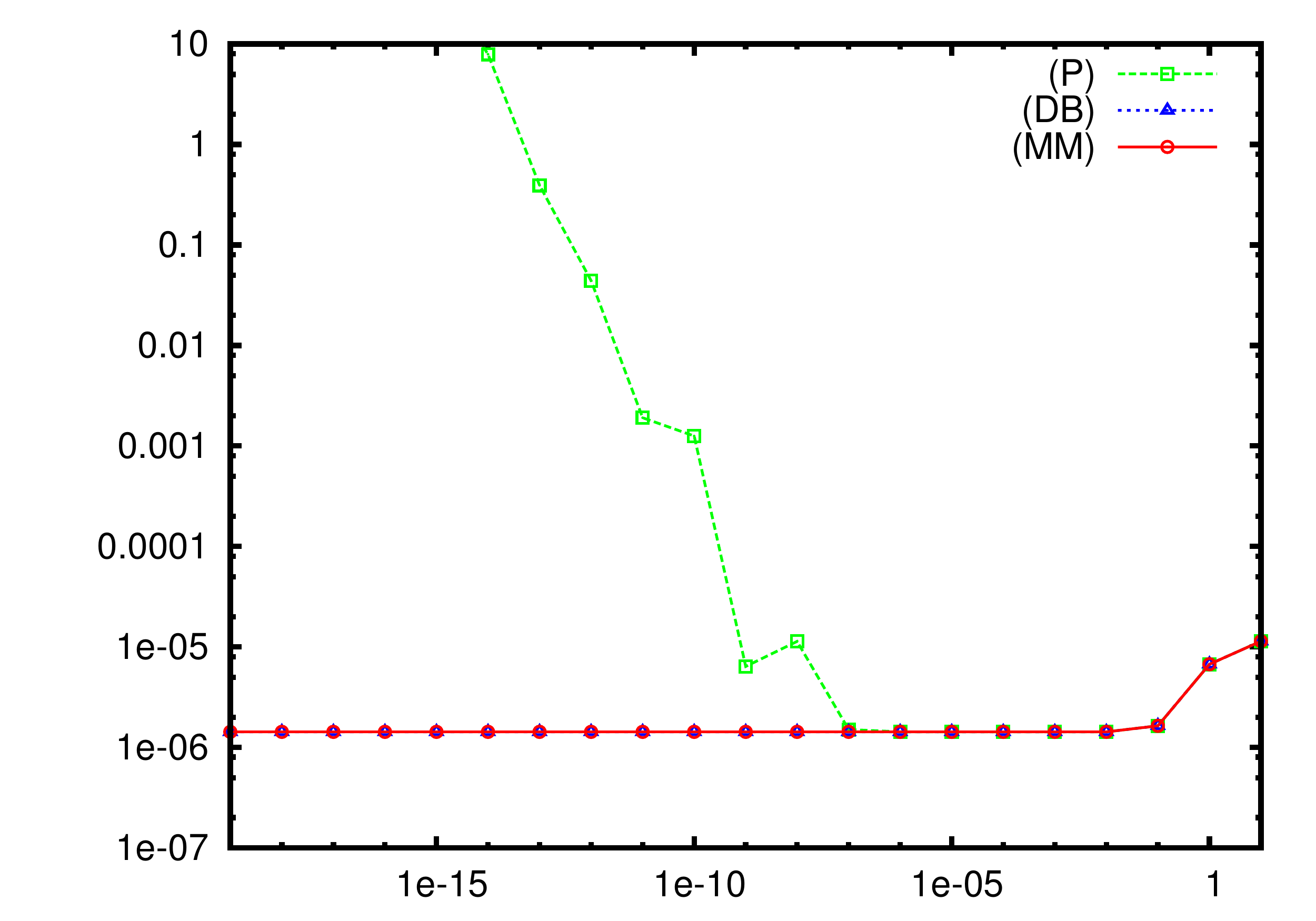}}
  \subfigure[$H^{1}$ error for a grid with $100\times 100$ points.]
  {\includegraphics[angle=-90,width=\xxxa]{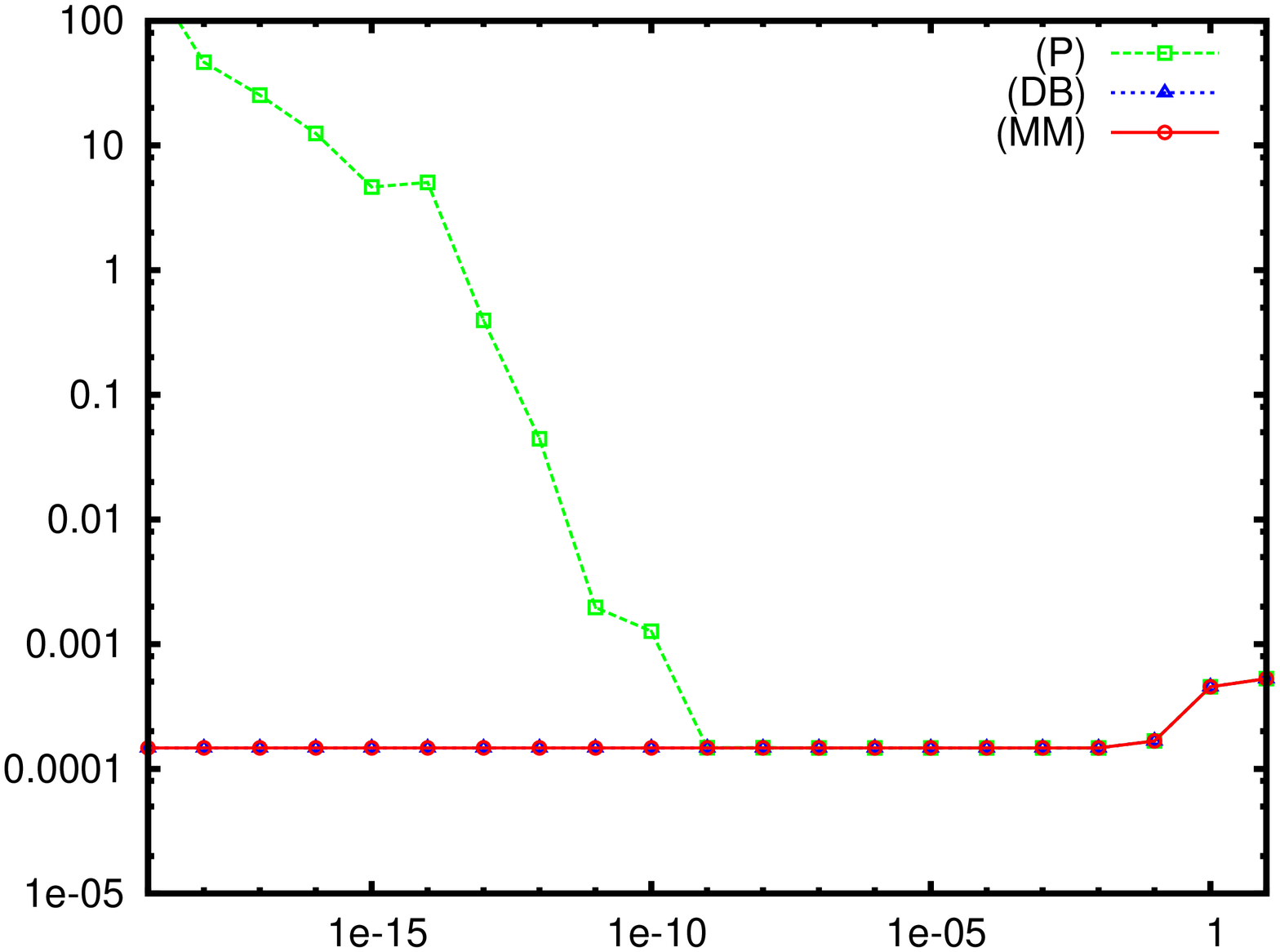}}

  \subfigure[$L^{2}$ error for a grid with $200\times 200$ points.]
  {\includegraphics[angle=-90,width=\xxxa]{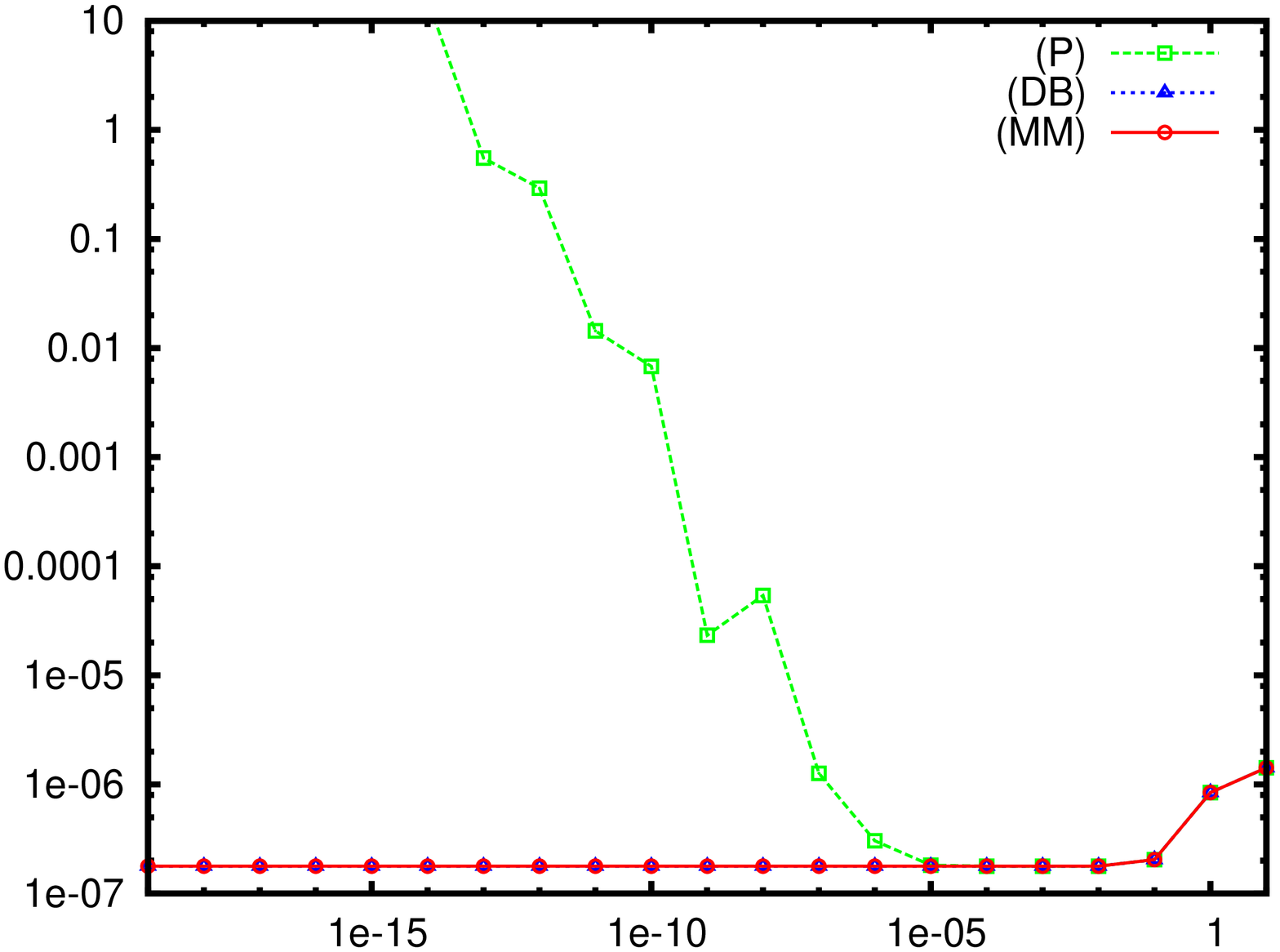}}
  \subfigure[$H^{1}$ error for a grid with $200\times 200$ points.]
  {\includegraphics[angle=-90,width=\xxxa]{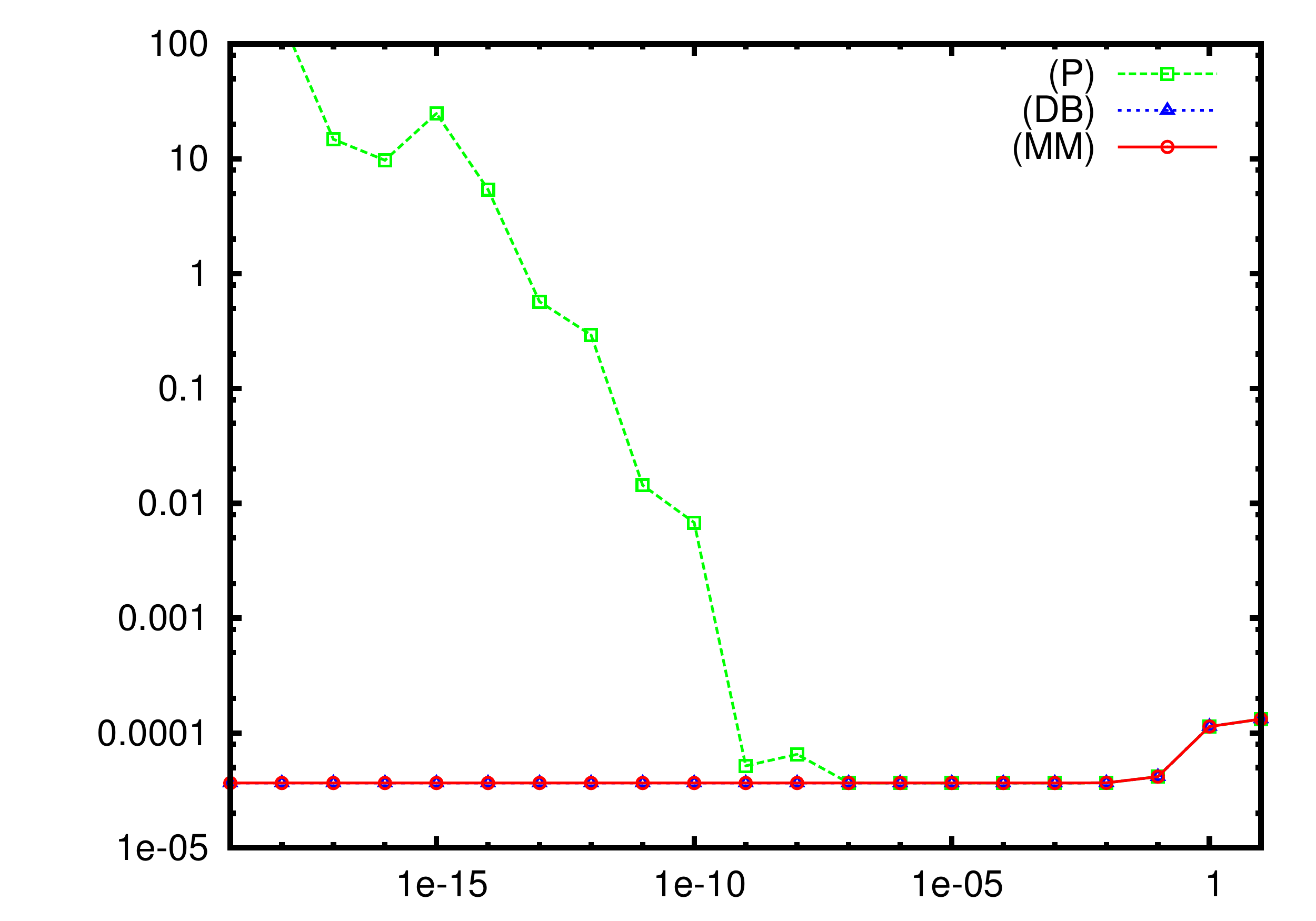}}

  \caption{Relative $L^{2}$ (left column) and $H^{1}$ (right column)
    errors between the exact solution $u^{\varepsilon }$ and the
    computed numerical solutions $u_M$ (MM), $u_D$ (DB), $u_P$ (P) for
    the test case with constant $b$. The error is plotted as a
    function of the parameter $\eps$ and for three different
    mesh-sizes.}
  \label{fig:error_ec_bc}
\end{figure}

\begin{table}
  \centering
  \begin{tabular}{|c||c|c||c|c||c|c|}
    \hline
    \multirow{2}{*}{$\varepsilon$} &
    \multicolumn{2}{|c||}{\rule{0pt}{2.5ex}MM scheme} 
    & \multicolumn{2}{|c||}{DB scheme} 
    & \multicolumn{2}{|c|}{Singular Perturbation scheme} \\
    \cline{2-7} 
    &\rule{0pt}{2.5ex}
    $L^2$ error & $H^1$ error & $L^2$ error & $H^1$ error & $L^2$ error & $H^1$ error  \\
    \hline\hline\rule{0pt}{2.5ex}
    10  & 
    $7.2\times 10^{-6}$ & $4.7\times 10^{-3}$ &
    $7.2\times 10^{-6}$ & $4.7\times 10^{-3}$ &
    $7.2\times 10^{-6}$ & $4.7\times 10^{-3}$ 
    \\
    \hline\rule{0pt}{2.5ex} 
    1 & 
    $7.3\times 10^{-7}$ & $4.7\times 10^{-4}$ &
    $7.3\times 10^{-7}$ & $4.7\times 10^{-4}$ &
    $7.3\times 10^{-7}$ & $4.7\times 10^{-4}$ 
    \\
    \hline\rule{0pt}{2.5ex} 
    $10^{-1}$&
    $1.47\times 10^{-7}$ & $9.6\times 10^{-5}$ &
    $1.47\times 10^{-7}$ & $9.6\times 10^{-5}$ &
    $1.45\times 10^{-7}$ & $9.4\times 10^{-5}$ 
    \\
    \hline\rule{0pt}{2.5ex} 
    $10^{-4}$&
    $1.28\times 10^{-7}$ & $8.3\times 10^{-5}$ &
    $1.28\times 10^{-7}$ & $8.3\times 10^{-5}$ &
    $1.26\times 10^{-7}$ & $8.2\times 10^{-5}$ 
    \\
    \hline\rule{0pt}{2.5ex} 
    $10^{-6}$&
    $1.28\times 10^{-7}$ & $8.3\times 10^{-5}$ &
    $1.28\times 10^{-7}$ & $8.3\times 10^{-5}$ &
    $5.9\times 10^{-7}$ & $8.2\times 10^{-5}$ 
    \\
    \hline\rule{0pt}{2.5ex} 
    $10^{-10}$&
    $1.28\times 10^{-7}$ & $8.3\times 10^{-5}$ &
    $1.28\times 10^{-7}$ & $8.3\times 10^{-5}$ &
    $9.9\times 10^{-3}$ & $3.12\times 10^{-2}$ 
    \\
    \hline\rule{0pt}{2.5ex} 
    $10^{-15}$&
    $1.28\times 10^{-7}$ & $8.3\times 10^{-5}$ &
    $1.28\times 10^{-7}$ & $8.3\times 10^{-5}$ &
    $7.1\times 10^{-1}$ & $2.23\times 10^{0}$ 
    \\
    \hline
  \end{tabular}
  \caption{Comparison between the Micro-Macro scheme, the
    Duality-Based reformulation and the Singular
    Perturbation model for $h=0.005$ (200 mesh points
    in each direction) and constant $b$: absolute $L^2$-errors and
    $H^1$-errors, for different $\eps$-values.}
  \label{tab:error}
\end{table}

In Figure \ref{fig:error_ec_bc} we plotted the absolute errors (in the
$L^2$ resp. $H^1$-norms) between the numerical solutions obtained with
one of the three methods and the exact solution, and this, as a
function of the parameter $\eps$ and for several mesh-sizes. In Table
\ref{tab:error}, we specified the error values for one fixed grid and
several $\eps$-values. One observes that the Singular Perturbation
finite element approximation is accurate only for $\varepsilon$ bigger
than some critical value $\varepsilon _P$ while the MM-scheme and the
DB-scheme are both accurate independently on $\eps$ and
give similar results. \\
The order of convergence for all three methods is three in the
$L^2$-norm and two in the $H^1$-norm, which is an optimal result for
$\mathbb Q_2$ finite elements. The convergence of the MM-scheme is
presented on Tables
\ref{tab:conv_e1} and \ref{tab:conv_e-100}. \\
Furthermore the condition number of the MM-scheme is bounded by an
$\eps$ independent constant and coincides with the condition number of
DB-scheme for $\eps <
0.1$. See Figure \ref{fig:cond} for the plots.
\begin{table}
  \centering
  \begin{tabular}{|c||c|c|c|c|c|}
    \hline\rule{0pt}{2.5ex}
    method & \# rows & \# non zero & time & $L^{2}$-error & $H^{1}$-error\\
    \hline
    \hline\rule{0pt}{2.5ex}
    MM &
    $20\times 10^{3}$ &
    $623\times 10^{3}$ &
    $1.156$ s &
    $1.19\times 10^{-6}$ &
    $1.47\times 10^{-4}$
    \\
    \hline\rule{0pt}{2.5ex}
    DB &
    $50\times 10^{3}$ &
    $1563\times 10^{3}$ &
    $7.405$ s &
    $1.19\times 10^{-6}$ &
    $1.47\times 10^{-4}$
    \\
    \hline\rule{0pt}{2.5ex}
    P &
    $10\times 10^{3}$ &
    $255\times 10^{3}$ &
    $0.501$ s &
    $1.19\times 10^{-6}$ &
    $1.47\times 10^{-4}$
    \\
    \hline
  \end{tabular}
  \caption{Comparison between the Micro-Macro AP-scheme (MM),
    the Duality-Based AP-scheme (DB) and the
    Singular Perturbation model (P) for 
    $h=0.01$ (100 mesh points in each direction) and fixed $\varepsilon =
    10^{-6}$: matrix size, number of nonzero elements, average
    computational time and relative error in $L^{2}$ and $H^{1}$ norms.}
  \label{tab:time}
\end{table}

\begin{table}
  \centering
  \begin{tabular}{|c||c|c|c|c|}
    \hline\rule{0pt}{2.5ex}
    $h$ & $L^{2}$-error in $u$ & $H^{1}$-error in $u$& $L^{2}$-error
    in $q$& $H^{1}$-error in $q$\\
    \hline
    \hline\rule{0pt}{2.5ex}
    0.1 &
    $5.7\times 10^{-3}$ &
    $1.86\times 10^{-1}$ &
    $5.7\times 10^{-3}$ &
    $1.86\times 10^{-1}$ 
    \\
    \hline\rule{0pt}{2.5ex}
    0.05 &
    $7.3\times 10^{-4}$ &
    $4.7\times 10^{-2}$ &
    $7.3\times 10^{-4}$ &
    $4.7\times 10^{-2}$ 
    \\
    \hline\rule{0pt}{2.5ex}
    0.025 &
    $9.1\times 10^{-5}$ &
    $1.18\times 10^{-2}$ &
    $9.1\times 10^{-5}$ &
    $1.18\times 10^{-2}$ 
    \\
    \hline\rule{0pt}{2.5ex}
    0.0125 &
    $1.14\times 10^{-5}$ &
    $2.96\times 10^{-3}$ &
    $1.14\times 10^{-5}$ &
    $2.96\times 10^{-3}$ 
    \\
    \hline\rule{0pt}{2.5ex}
    0.00625 &
    $1.43\times 10^{-6}$ &
    $7.4 \times 10^{-4}$ &
    $1.43\times 10^{-6}$ &
    $7.4 \times 10^{-4}$ 
    \\
    \hline\rule{0pt}{2.5ex}
    0.003125 &
    $1.78\times 10^{-7}$ &
    $1.85\times 10^{-4}$ &
    $1.78\times 10^{-7}$ &
    $1.85\times 10^{-4}$ 
    \\
    \hline\rule{0pt}{2.5ex}
    0.0015625 &
    $2.23\times 10^{-8}$ &
    $4.6\times 10^{-5}$ &
    $2.23\times 10^{-8}$ &
    $4.6\times 10^{-5}$ 
    \\
    \hline
  \end{tabular}
  \caption{The absolute error of $u$ and $q$ in $L^{2}$ and $H^{1}$-norms 
    for different mesh sizes and $\eps =1$. Used discretization
    method:  Micro-Macro scheme (MM).
  }
  \label{tab:conv_e1}
\end{table}
\begin{table}
  \centering
  \begin{tabular}{|c||c|c|c|c|}
    \hline\rule{0pt}{2.5ex}
    $h$ & $L^{2}$-error in $u$ & $H^{1}$-error in $u$& $L^{2}$-error
    in $q$& $H^{1}$-error in $q$\\
    \hline
    \hline\rule{0pt}{2.5ex}
    0.1 &
    $1.00\times 10^{-3}$ &
    $3.25\times 10^{-2}$ &
    $5.7\times 10^{-3}$ &
    $1.86\times 10^{-1}$ 
    \\
    \hline\rule{0pt}{2.5ex}
    0.05 &
    $1.26\times 10^{-4}$ &
    $8.2\times 10^{-3}$ &
    $7.3\times 10^{-4}$ &
    $4.7\times 10^{-2}$ 
    \\
    \hline\rule{0pt}{2.5ex}
    0.025 &
    $1.58\times 10^{-5}$ &
    $2.04\times 10^{-3}$ &
    $9.1\times 10^{-5}$ &
    $1.18\times 10^{-2}$
    \\
    \hline\rule{0pt}{2.5ex}
    0.0125 &
    $1.97\times 10^{-6}$ &
    $5.1\times 10^{-4}$ &
    $1.14\times 10^{-5}$ &
    $2.96\times 10^{-3}$
    \\
    \hline\rule{0pt}{2.5ex}
    0.00625 &
    $2.46\times 10^{-7}$ &
    $1.28\times 10^{-4}$ &
    $1.43\times 10^{-6}$ &
    $7.4 \times 10^{-4}$
    \\
    \hline\rule{0pt}{2.5ex}
    0.003125 &
    $3.1 \times 10^{-8}$ &
    $3.2 \times 10^{-5}$ &
    $2.59\times 10^{-5}$ &
    $2.96\times 10^{-2}$
    \\
    \hline\rule{0pt}{2.5ex}
    0.0015625 &
    $4.1\times 10^{-9}$ &
    $8.0\times 10^{-6}$ &
    $4.0\times 10^{-4}$ &
    $9.1\times 10^{-1}$ 
    \\
    \hline
  \end{tabular}
  \caption{The absolute error of $u$ and $q$ in $L^{2}$ and $H^{1}$-norms 
    for different mesh sizes and $\eps =10^{-100}$. Used
    discretization method: Micro-Macro scheme (MM). 
  }
  \label{tab:conv_e-100}
\end{table}

\def\xxxa{0.45\textwidth}
\begin{figure}[!ht] 
  \centering
  \subfigure[mesh size: $100\times 100$ points]
  {\includegraphics[angle=-90,width=\xxxa]{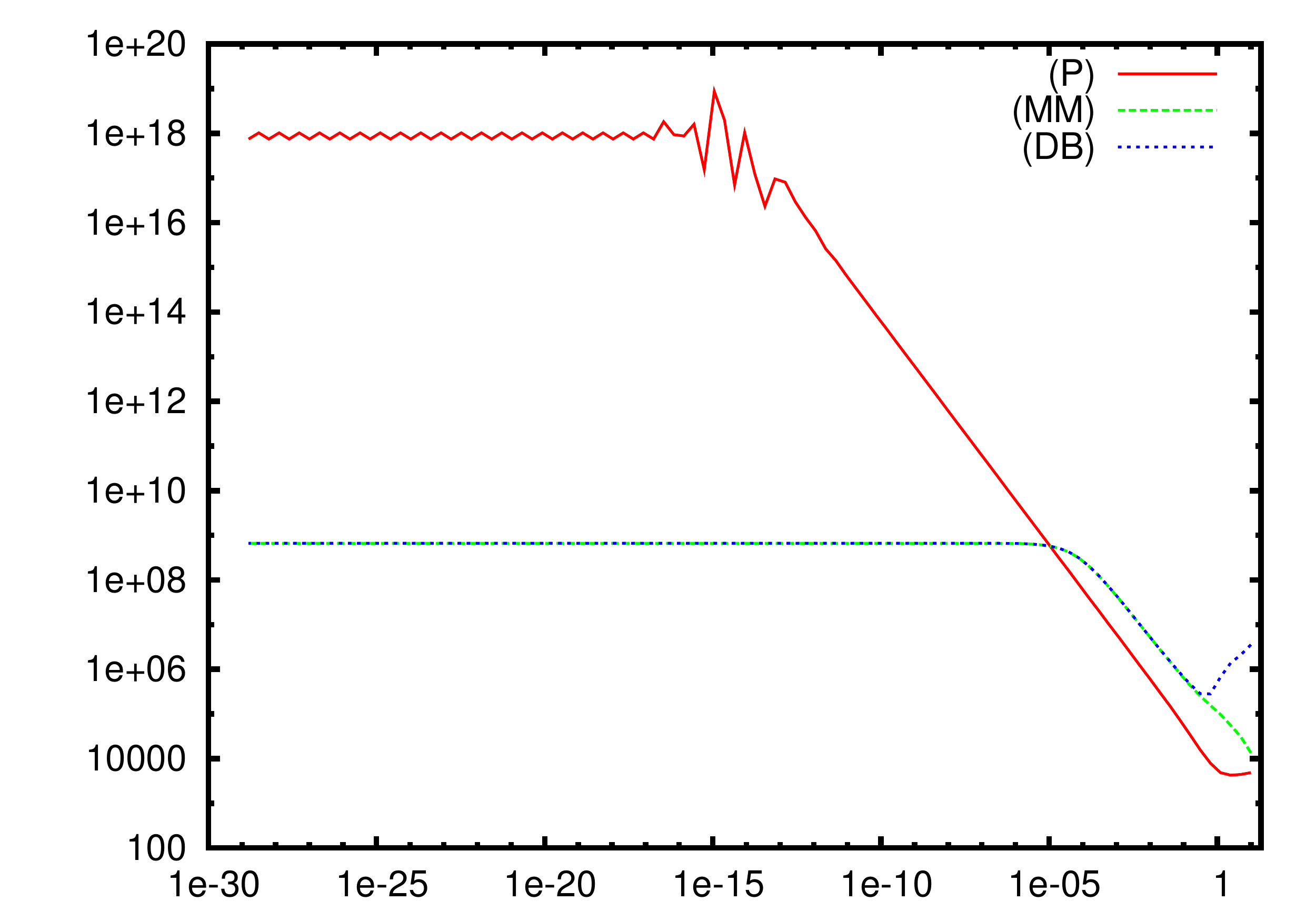}}
  \subfigure[mesh size: $400\times 400$ points]
  {\includegraphics[angle=-90,width=\xxxa]{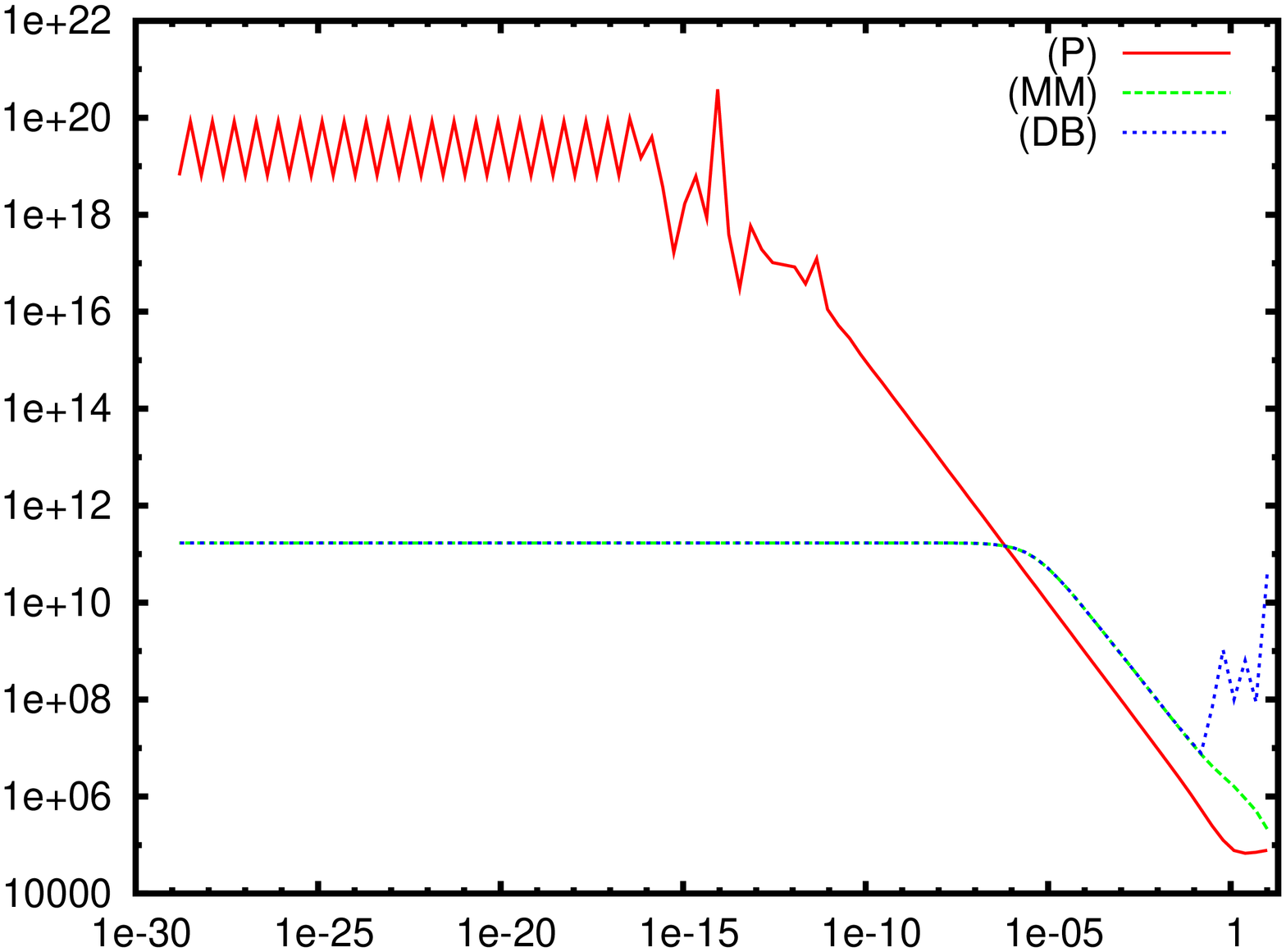}}

  \caption{Condition number estimate provided by the MUMPS solver for
    the (MM), (DB) and (P) schemes.}
  \label{fig:cond}
\end{figure}

Since both Asymptotic Preserving models (DB/MM) give the same accuracy in this
test case, it is worthwhile to compare the computational resources
required to obtain this same results. The computational time and the
matrix sizes required to solve the problem for fixed $\varepsilon$ and
$h$ are given in the Table \ref{tab:time}. As expected, the MM-scheme is more
efficient than the DB-approach. The average computational
time of the MM-scheme is approximately $6.4$ times smaller as compared
to the DB-simulation time. It should be noted that the
Singular Perturbation model is approximately $2.3$ faster than the new
MM-method. However, the applicability of the Singular Perturbation scheme is
limited to sufficiently large values of $\eps$.

\subsubsection{2D test case, constant $\varepsilon$, non-uniform and non-aligned $b$-field}\label{sec:tc_bv}

We now focus our attention on the original feature of the here
introduced numerical method, namely its ability to treat nonuniform
$b$ fields. In this section we present numerical simulations performed
for a variable field $b$.

\def\xxxa{0.45\textwidth}
\begin{figure}[!ht] 
  \centering
  \subfigure[$L^{2}$ error for a grid with $50\times 50$ points.]
  {\includegraphics[angle=-90,width=\xxxa]{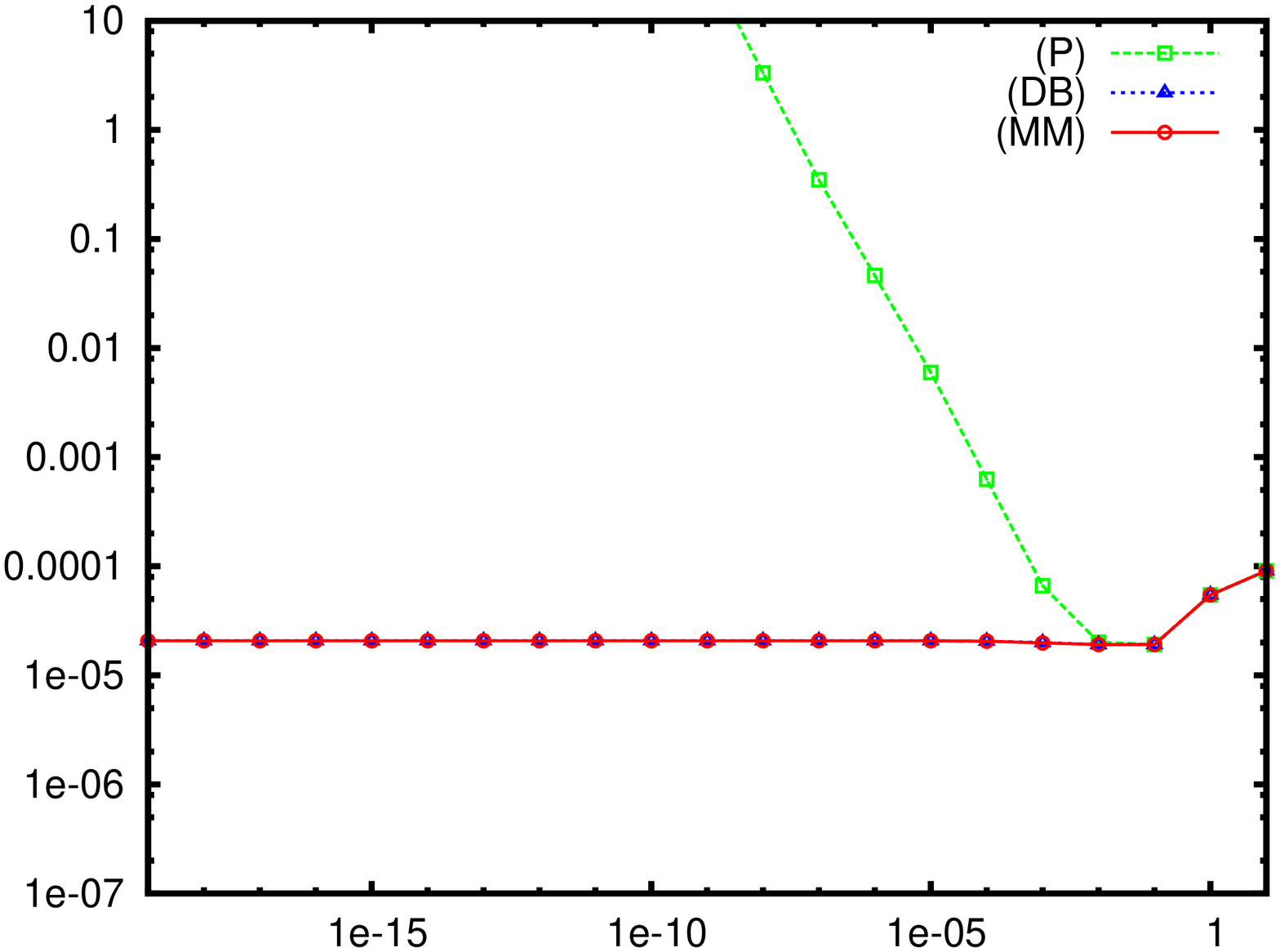}}
  \subfigure[$H^{1}$ error for a grid with $50\times 50$ points.]
  {\includegraphics[angle=-90,width=\xxxa]{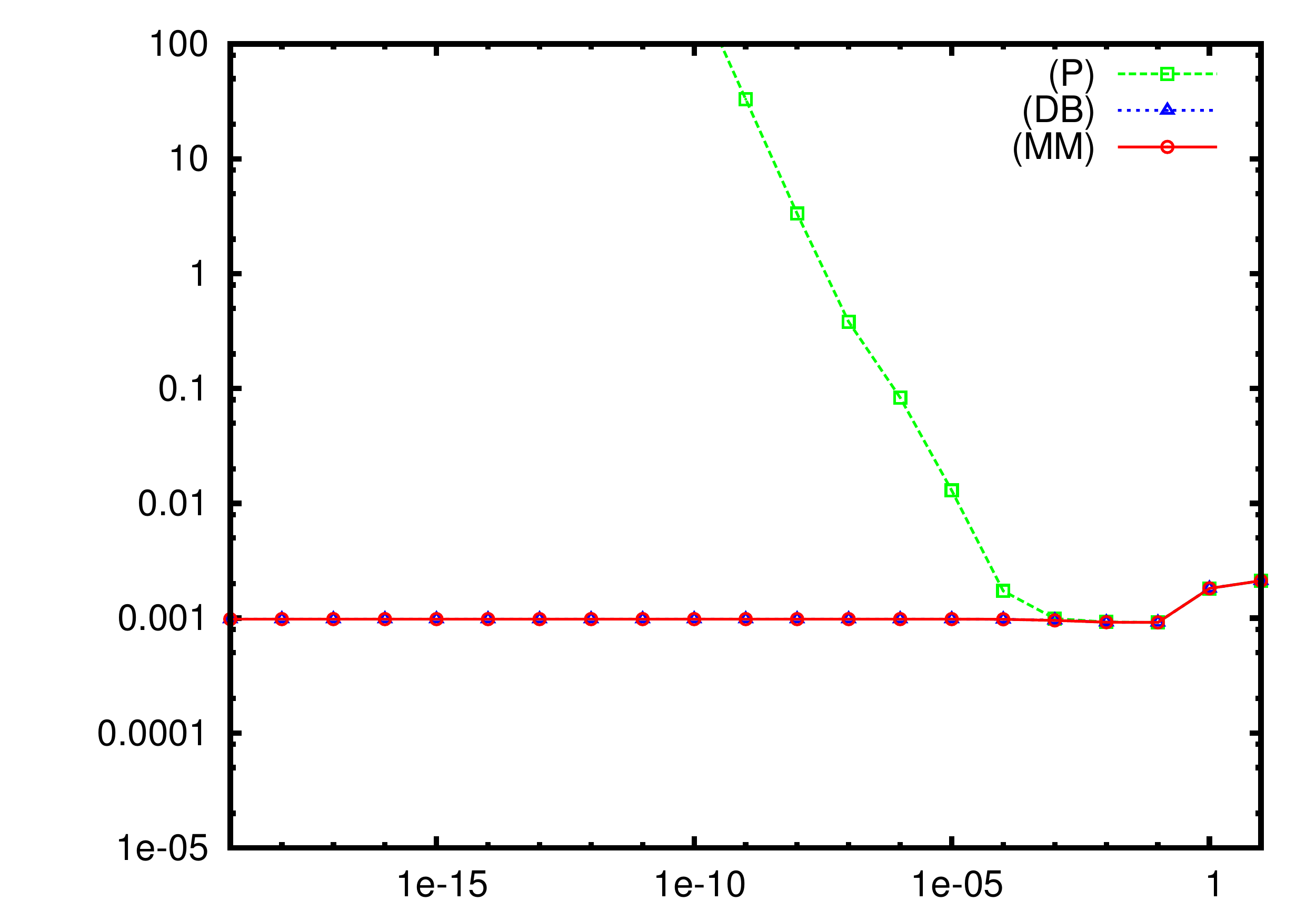}}

  \subfigure[$L^{2}$ error for a grid with $100\times 100$ points.]
  {\includegraphics[angle=-90,width=\xxxa]{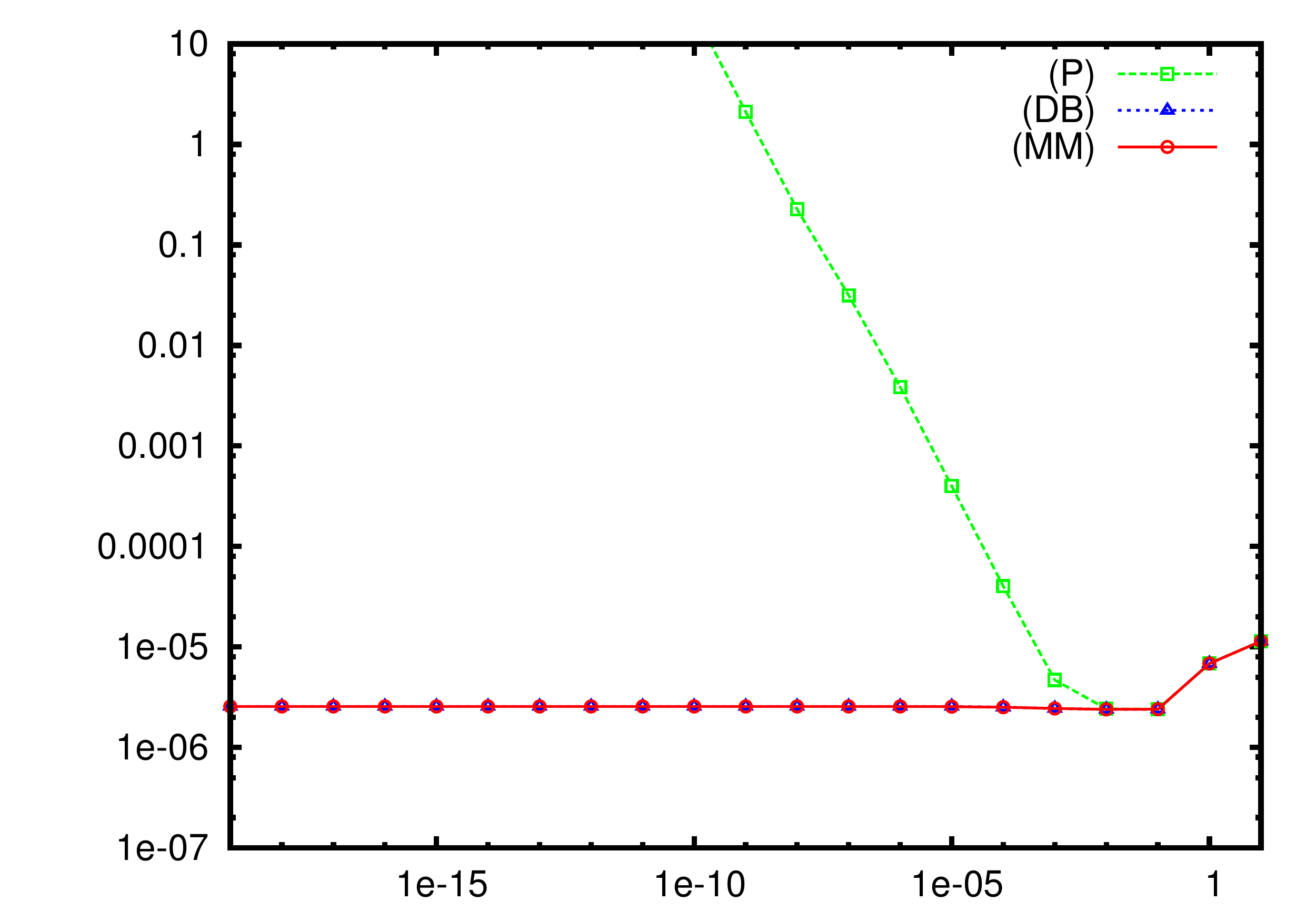}}
  \subfigure[$H^{1}$ error for a grid with $100\times 100$ points.]
  {\includegraphics[angle=-90,width=\xxxa]{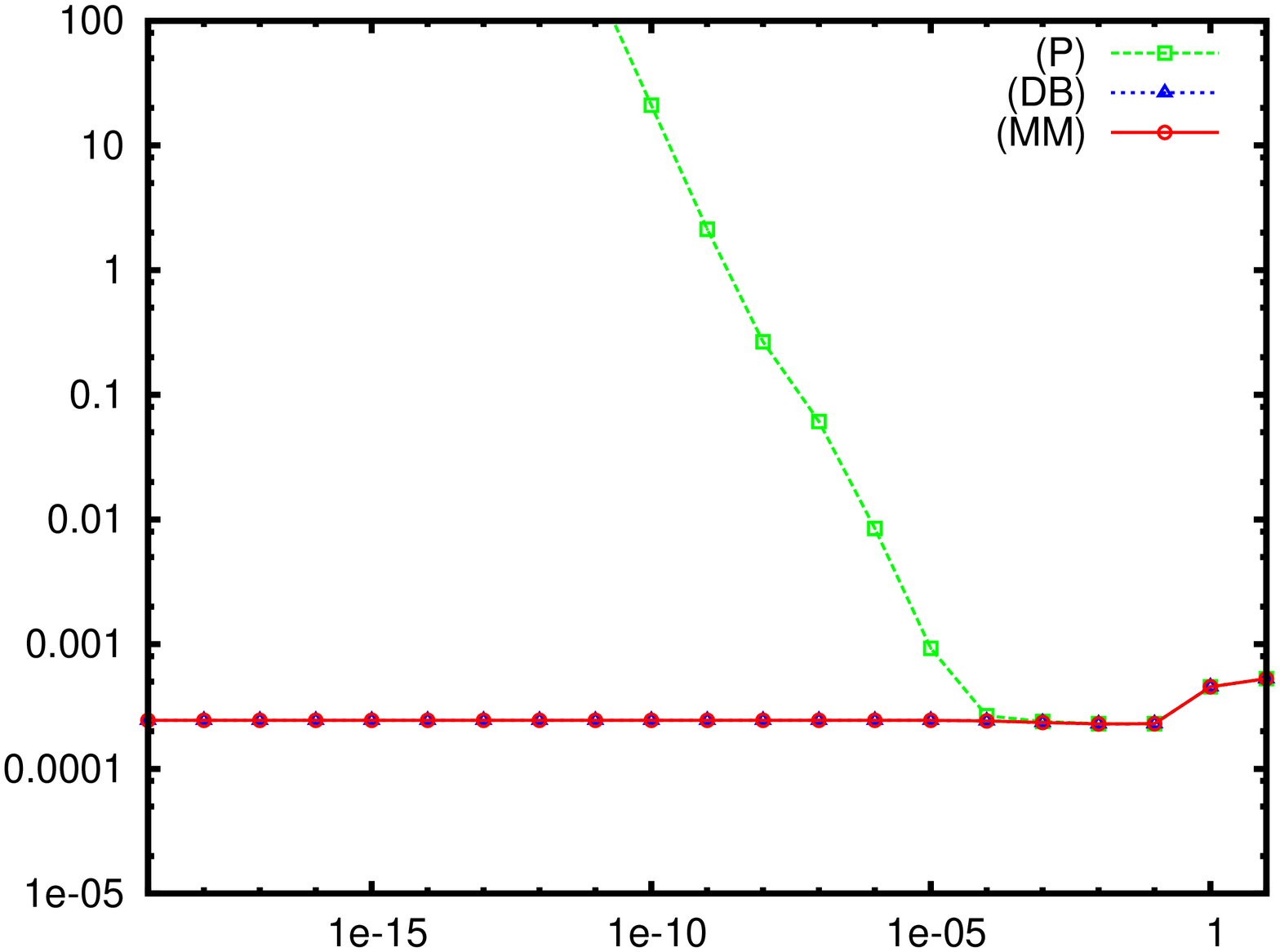}}

  \subfigure[$L^{2}$ error for a grid with $200\times 200$ points.]
  {\includegraphics[angle=-90,width=\xxxa]{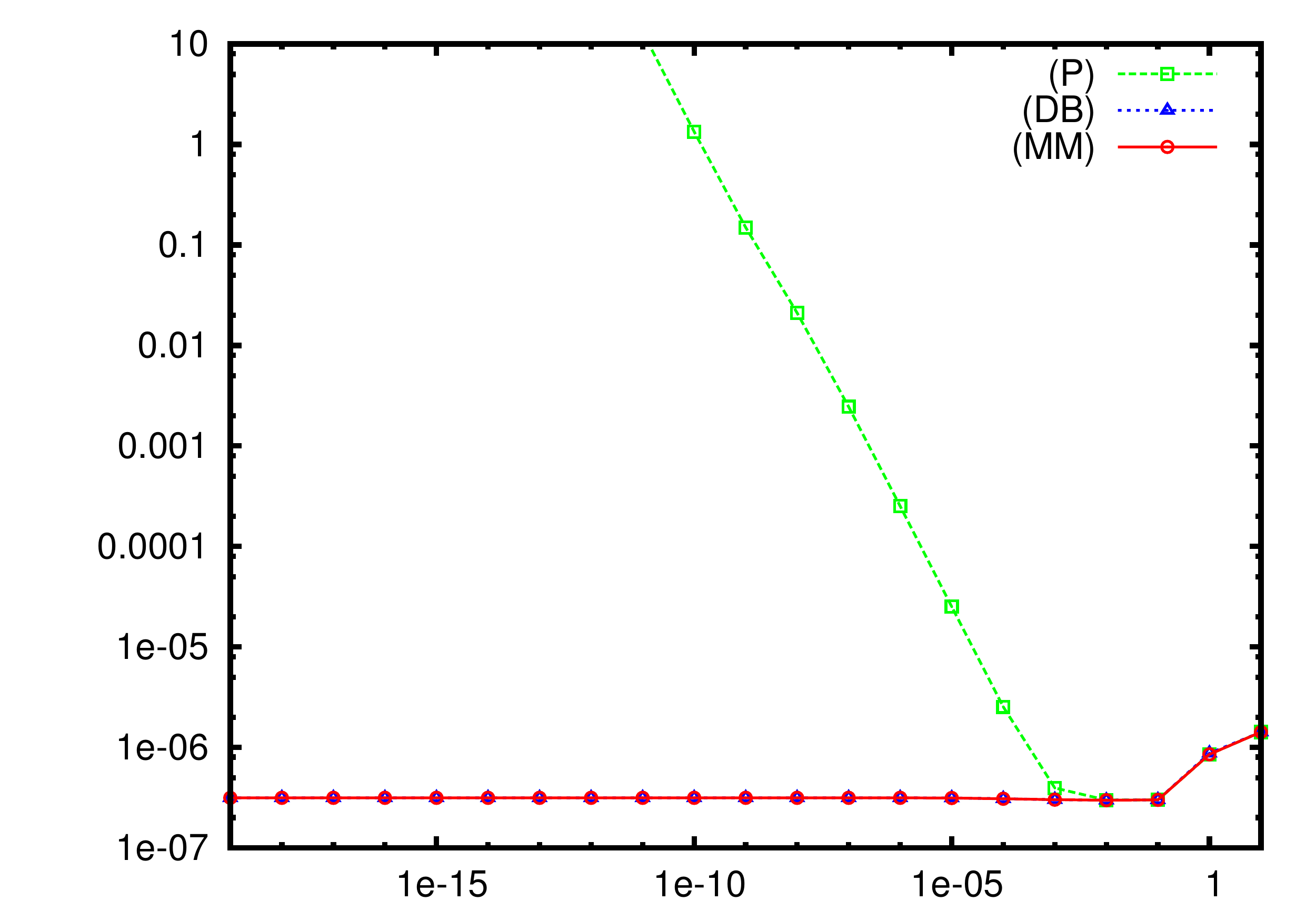}}
  \subfigure[$H^{1}$ error for a grid with $200\times 200$ points.]
  {\includegraphics[angle=-90,width=\xxxa]{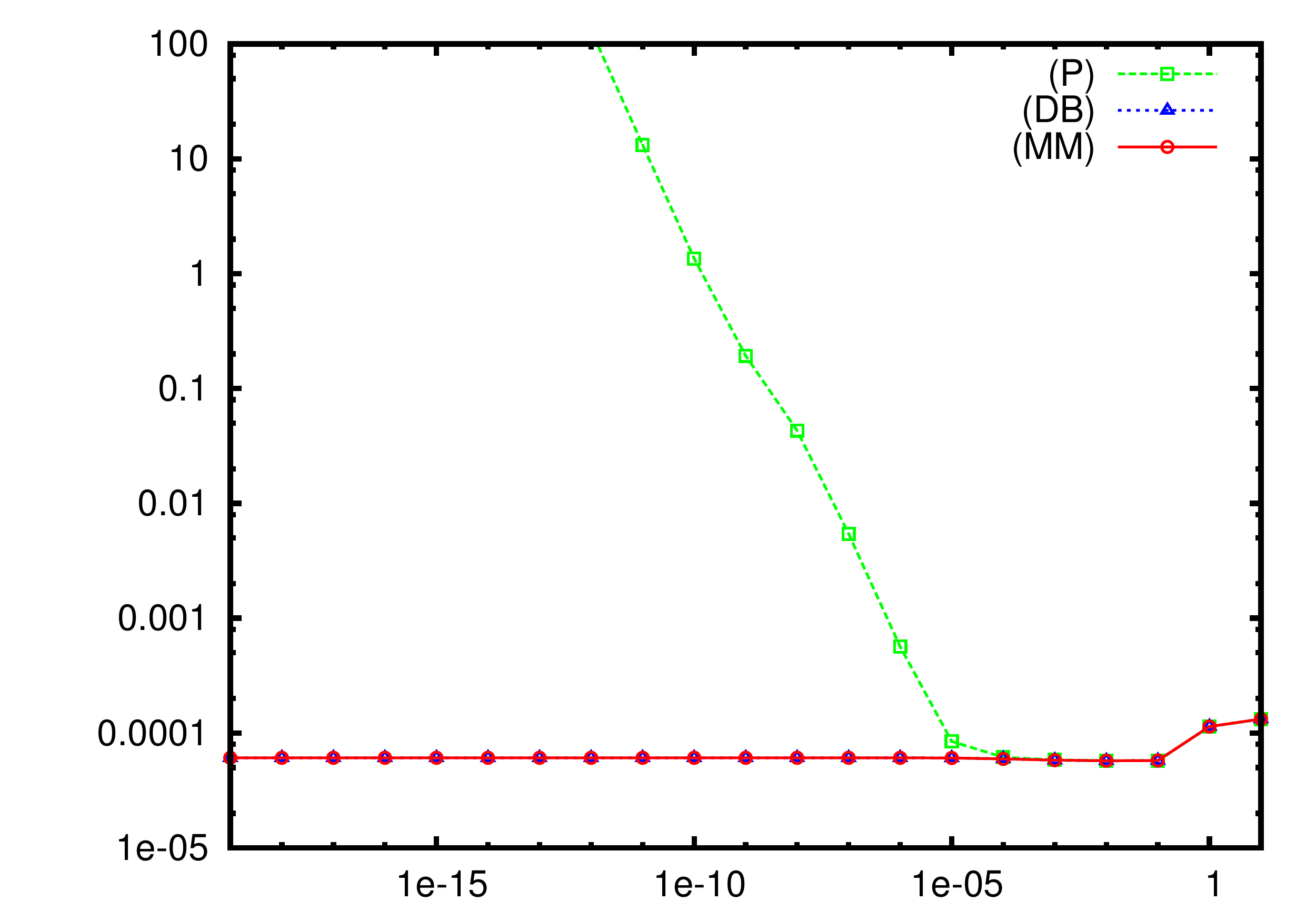}}

  \caption{Relative $L^{2}$ (left column) and $H^{1}$ (right column)
    errors between the exact solution $u^{\varepsilon }$ and the
    computed solution $u_M$ (MM), $u_D$ (DB), $u_P$ (P) for
    the test case with variable $b$. Plotted are the errors as a
    function of the small parameter $\varepsilon $, for three different
    meshes.}
  \label{fig:error_ec_bv}
\end{figure}

First, let us construct a numerical test case. Finding an analytical
solution for an arbitrary $b$-field presents a considerable
difficulty. In the previous paper \cite{DDLNN}, we presented a way to
find such a solution. Let us recall briefly how to do.  First, we
choose a limit solution
\begin{gather}
  u^{0} = \sin \left(\pi y +\alpha (y^2-y)\cos (\pi x) \right)
  \label{eq:J79a},
\end{gather}
where $\alpha $ is a numerical constant aimed to control the
variations of $b$. For $\alpha =0$, the limit solution of the previous
section is obtained. The limit solution for $\alpha =2$ is shown in
Figure \ref{fig:limit}.
\begin{figure}[t] 
  \centering
  \def\xxxa{0.45\textwidth}
  \includegraphics[width=\xxxa]{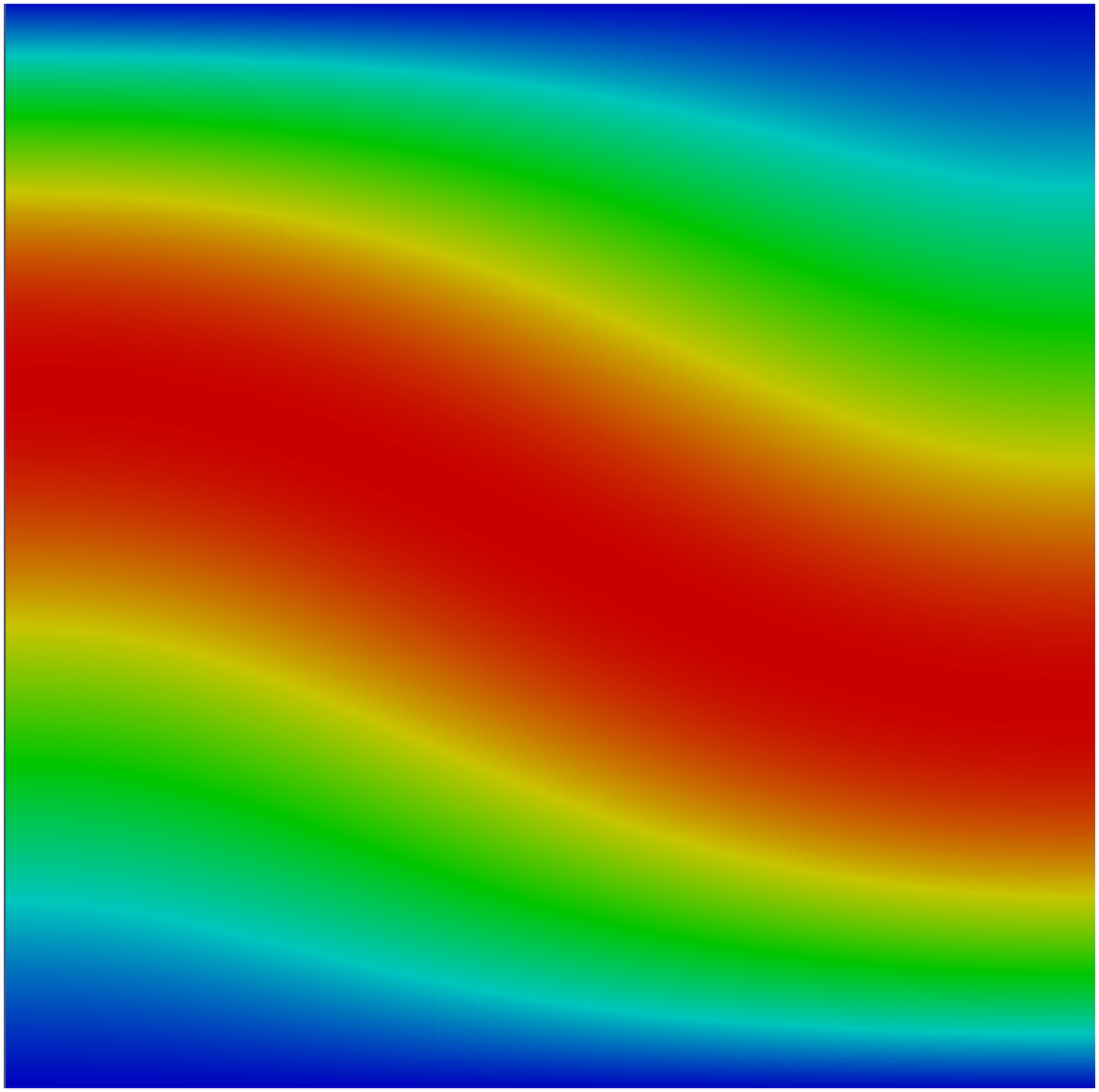}
  \caption{The limit solution for the test case with variable $b$.}
  \label{fig:limit}
\end{figure}
Since $u^{0}$ is a limit solution, it is constant along the $b$
field lines. Therefore we can determine the $b$ field using the
following implication
\begin{gather}
  \nabla_{\parallel} u^{0} = 0 \quad \Rightarrow \quad
  b_x \frac{\partial u^{0}}{\partial x} +
  b_y \frac{\partial u^{0}}{\partial y} = 0\,,
  \label{eq:J89a}
\end{gather}
which yields for example
\begin{gather}
  b = \frac{B}{|B|}\, , \quad
  B =
  \left(
    \begin{array}{c}
      \alpha  (2y-1) \cos (\pi x) + \pi \\
      \pi \alpha  (y^2-y) \sin (\pi x)
    \end{array}
  \right)
  \label{eq:J99a}\,\quad. 
\end{gather}
Note that the field $B$, constructed in this way, satisfies
$\text{div} B = 0$, which is an important property in the framework of
plasma simulations. Furthermore, we have $B \neq 0$ in the
computational domain. Now, we choose $u^\varepsilon $ to be a
function that converges, as $\varepsilon \rightarrow 0$, to the limit
solution $u^{0}$, for example
\begin{gather*}
  u^{\varepsilon } = \sin \left(\pi y +\alpha (y^2-y)\cos (\pi
    x) \right) + \varepsilon \cos \left( 2\pi x\right) \sin \left(\pi
    y \right)
\end{gather*}
Finally, the force term is calculated, using the equation, i.e.
\begin{gather}
  f = - \nabla_\perp \cdot (A_\perp \nabla_\perp u^{\varepsilon })
  - \frac{1}{\varepsilon }\nabla_\parallel \cdot (A_\parallel \nabla_\parallel u^{\varepsilon })
  \nonumber.
\end{gather}

The simulation results are presented on Figure \ref{fig:error_ec_bv}.
Similarly to the previous test case, both, the Micro-Macro Asymptotic
Preserving approach and the Duality-Based Asymptotic Preserving
reformulation, give the same accuracy. Both methods converge with the
optimal rate in both $L^{2}$ and $H_1$ norms, independently of
$\eps$. The discretization of the Singular-Perturbation model is again
limited. The critical value $\eps_P$ is now of the order of $10^{-2}$
and seems mesh independent while in the uniform $b$ case this value
ranged from $10^{-9}$ to $10^{-5}$ depending on the mesh size.

It is worthwile to investigate the influence of the variations of the
$b$-field on the accuracy of the solution. It is particularly
interesting to find out the minimal number of mesh nodes per
characteristic length of $b$-variations required to obtain an
acceptable solution. For this, let us proceed as in the previous paper
\cite{DDLNN} and define a variant of the test case presented above.

Let $b = B/|B|$, with
\begin{gather}
  B =
  \left(
    \begin{array}{c}
      \alpha  (2y-1) \cos (m\pi x) + \pi \\
      m\pi \alpha  (y^2-y) \sin (m\pi x)
    \end{array}
  \right)\, ,
  \label{eq:Ji0a}
\end{gather}
$m$ being an integer. The limit solution $\phi^0$ and $\phi^{\varepsilon}$ are
chosen to be
\begin{gather}
  \phi^{0} = \sin \left(\pi y +\alpha (y^2-y)\cos (m\pi x) \right)
  \label{eq:Jj0a}, \\
  \phi^{\varepsilon } = \sin \left(\pi y +\alpha (y^2-y)\cos (m\pi
    x) \right) + \varepsilon \cos \left( 2\pi x\right) \sin \left(\pi
    y \right)
  \label{eq:Jk0a}.
\end{gather}
As in \cite{DDLNN}, we perform numerical simulations on a fixed $400
\times 400$ grid ($h=0.0025$) and vary $m$. We obtain the same
results: the relative error in the $L^{2}$-norm, defined as
$\frac{||u^{\varepsilon} - u_M||_{L^{2}(\Omega)}}{||u_M||_{L^{2}(\Omega)}}$, is below $0.01$
for all tested values of $1 \leq m \leq 50$. The relative $H^1$-error
$\frac{||u^{\varepsilon} - u_M||_{H^{1}(\Omega)}}{||u_M||_{H^{1}(\Omega)}}$ exceeds the
critical value for $m > 25$. For $\varepsilon = 10^{-20}$ the maximal
$m$ for which the error is acceptable in both norms is $20$. The
minimal number of mesh points per period of $b$ variations is 40 in
the worst case, in order to obtain an $1\%$ relative error. The
results are summarized on Figures \ref{fig:bp_var}. The results prove
that the MM-scheme is precise even for strongly oscillating fields for
relatively small mesh sizes.

\def\xxxa{0.45\textwidth}
\begin{figure}[!ht]
  \centering
  \subfigure[$\varepsilon = 1$]
  {\includegraphics[angle=-90,width=\xxxa]{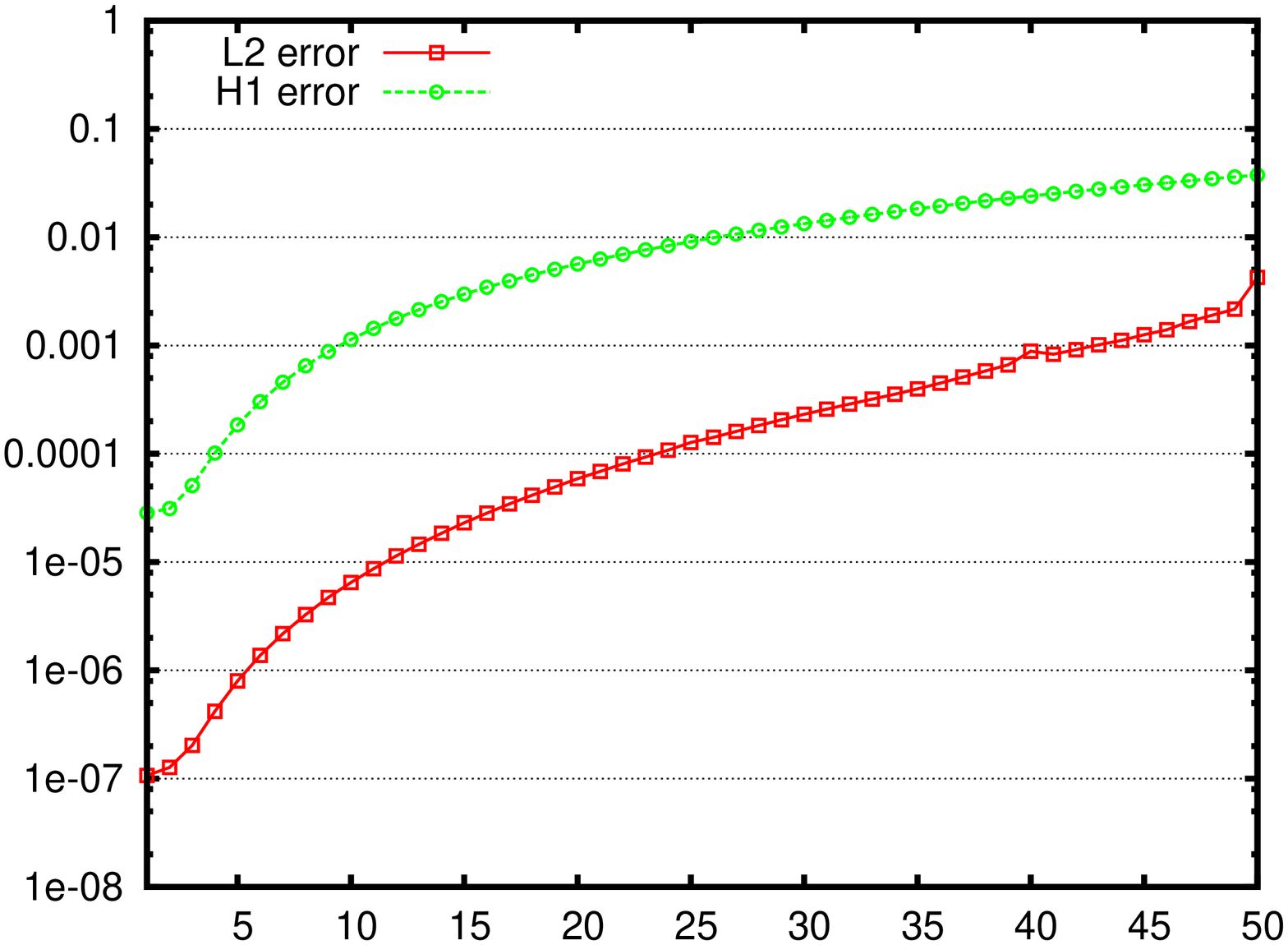}}
  \subfigure[$\varepsilon = 10^{-20}$]
  {\includegraphics[angle=-90,width=\xxxa]{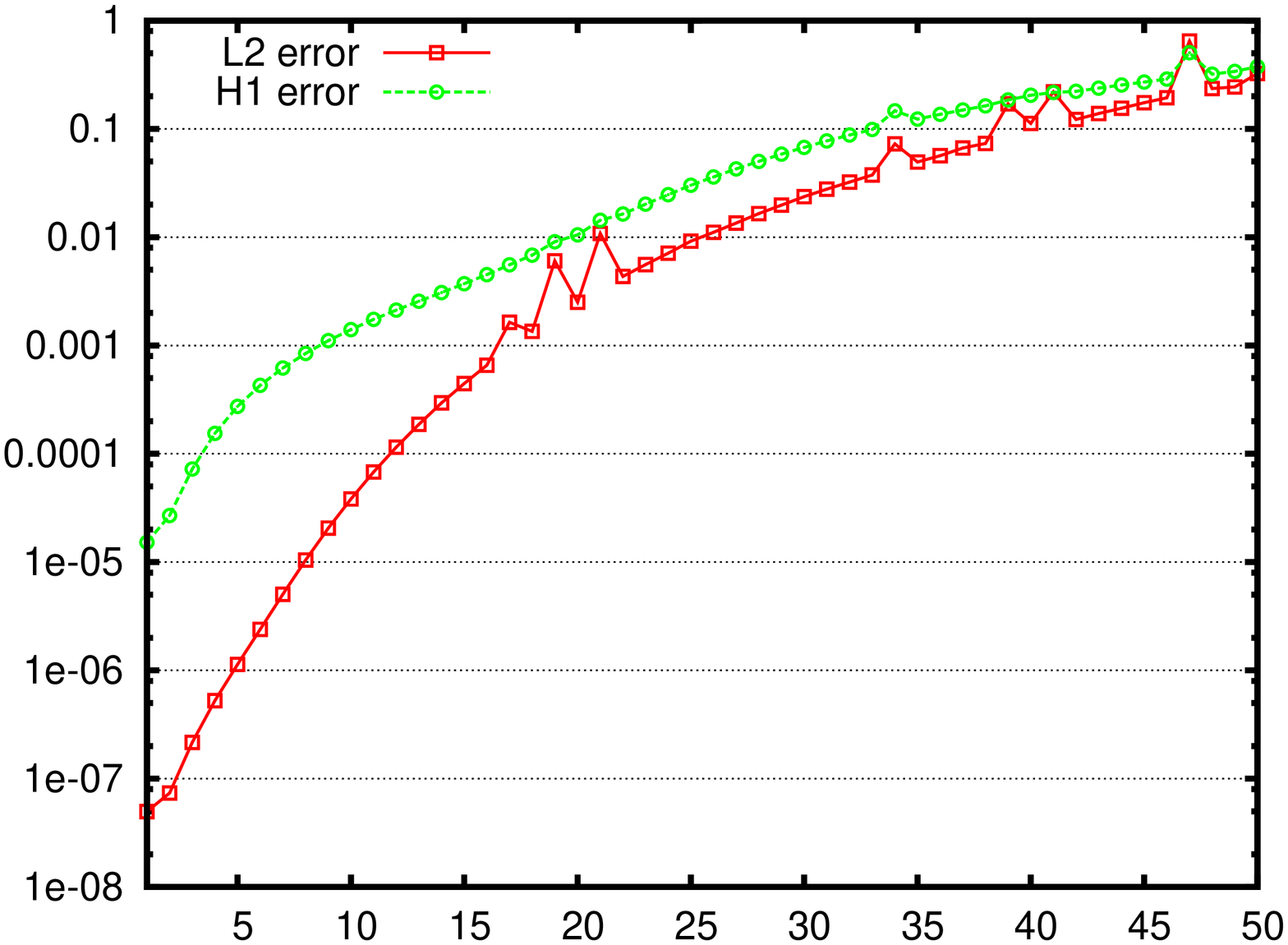}}

  \caption{Relative $L^{2}$ and $H^{1}$ errors between the exact
    solution $\phi^{\varepsilon }$ and the computed solution $\phi _M$
    (MM) for $h=0.0025$ (400 points in each direction) as a function of
    $m$ and for $\varepsilon = 1$ respectively  $10^{-20}$.}
  \label{fig:bp_var}
\end{figure}

\section{Generalization to the variable $\eps$ case} \label{GENE}

In this section we introduce a generalization of the Micro-Macro
scheme to the case of an anisotropy intensity $\eps$ variable in
space.  The anisotropy parameter $\eps$ is now a function of space
coordinates and takes values (within the computational domain) in the
interval $[\eps_{min},1]$. This setting requires a novel approach to
the problem since the original decomposition presented in \ref{sec:AP}
would yield a function $q$ that could take values of $O(1)$ in the
subdomain where $\eps$ is big and of $O(1/\eps_{min})$ in the
subdomain where $\eps$ is close to $\eps_{min}$.  To understand this,
let us consider the simple test of Section \ref{sec:tc_bc},
i.e. $u^\eps=\sin \left(\pi y \right) + \varepsilon \cos \left( 2\pi
  x\right) \sin \left(\pi y \right) $ with $b$ constant and aligned
with the $x$-axis. Let $\eps = 1$ for $x<0.5$ and $\eps = \eps_{min}$
for $x>0.5$.  In the decomposition $u = p + \eps q$ the function $p$
is constant in the direction of the anisotropy field $b$ and
$p|_{\Gamma_{in}} = \phi|_{\Gamma_{in}}$. In the mentioned test case
$p = 2 \sin \left(\pi y \right) $. For $x<0.5$ we obtain $q = \cos
\left(2\pi x\right) \sin \left(\pi y \right) - \sin \left(\pi y
\right) = O(1)$. However, for $x > 0.5$ we get $q = \cos \left(2\pi
  x\right) \sin \left(\pi y \right) - \frac{1}{\eps_{min}}\sin
\left(\pi y \right) = O(1/\eps_{min})$. Therefore the decomposition
$u=p+\eps q$ is not suitable in the case of variable $\eps$, the
function $q$ being no more bounded.

 The remedy to this deficiency is surprisingly simple. Instead of
introducing the decomposition $u=p+\eps q$ it suffices to define $q$
via the following relation 
\begin{gather}
  \nabla_\parallel q = \frac{1}{\eps} \nabla_\parallel u
  \label{eq:Joab}
\end{gather}
with again $q=0$ on $\Gamma_{in}$. This leads to almost the same
system as introduced in \ref{sec:AP}. The only change is that the variable $\eps(x)$ should be now inside the integral:
\begin{equation}
\left\{ 
\begin{array}{ll}
\ds \int_{\Omega }( A_{\perp }\nabla _{\perp }u^\eps)\cdot \nabla _{\perp
}v\,dx+\int_{\Omega } A_{||} \nabla _{||}q^\eps\cdot \nabla _{||}v\,dx=\int_{\Omega }fv\,dx, & \quad
\forall v\in {\cal V} \\[3mm]
\ds \int_{\Omega }A_{||} \nabla {}_{||}u^\eps\cdot \nabla {}_{||}w\,dx- \int_{\Omega }\varepsilon
A_{||} \nabla _{||}q^\eps\cdot \nabla _{||}w\,dx=0, & \quad \forall w\in {\cal L}\,.
\end{array}%
\right.  \label{PaEvar}
\end{equation}

\def\xxxa{0.45\textwidth}
\begin{figure}[!ht] 
  \centering
  \subfigure[$L^{2}$ error for a grid with $50\times 50$ points.]
  {\includegraphics[angle=-90,width=\xxxa]{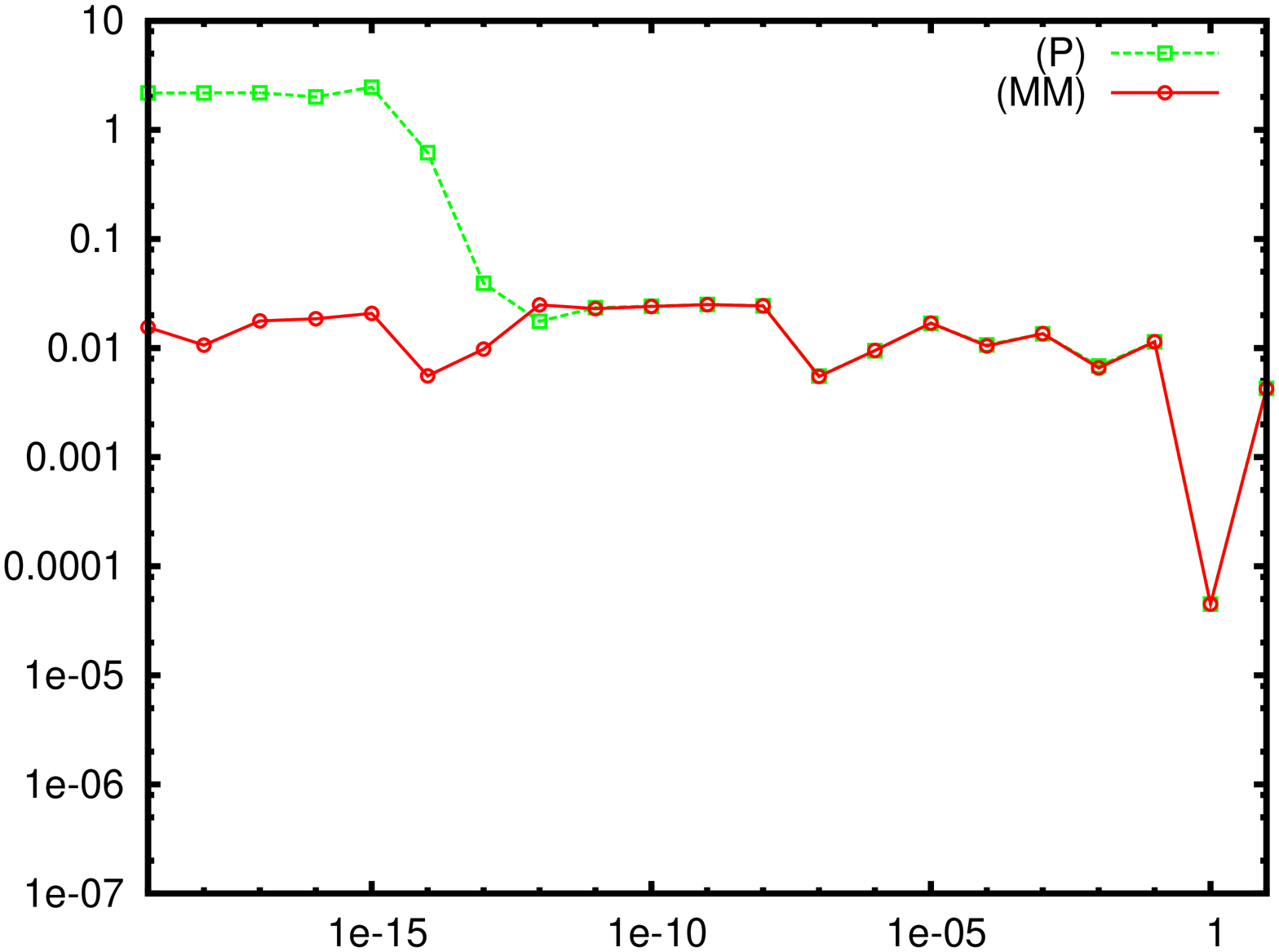}}
  \subfigure[$H^{1}$ error for a grid with $50\times 50$ points.]
  {\includegraphics[angle=-90,width=\xxxa]{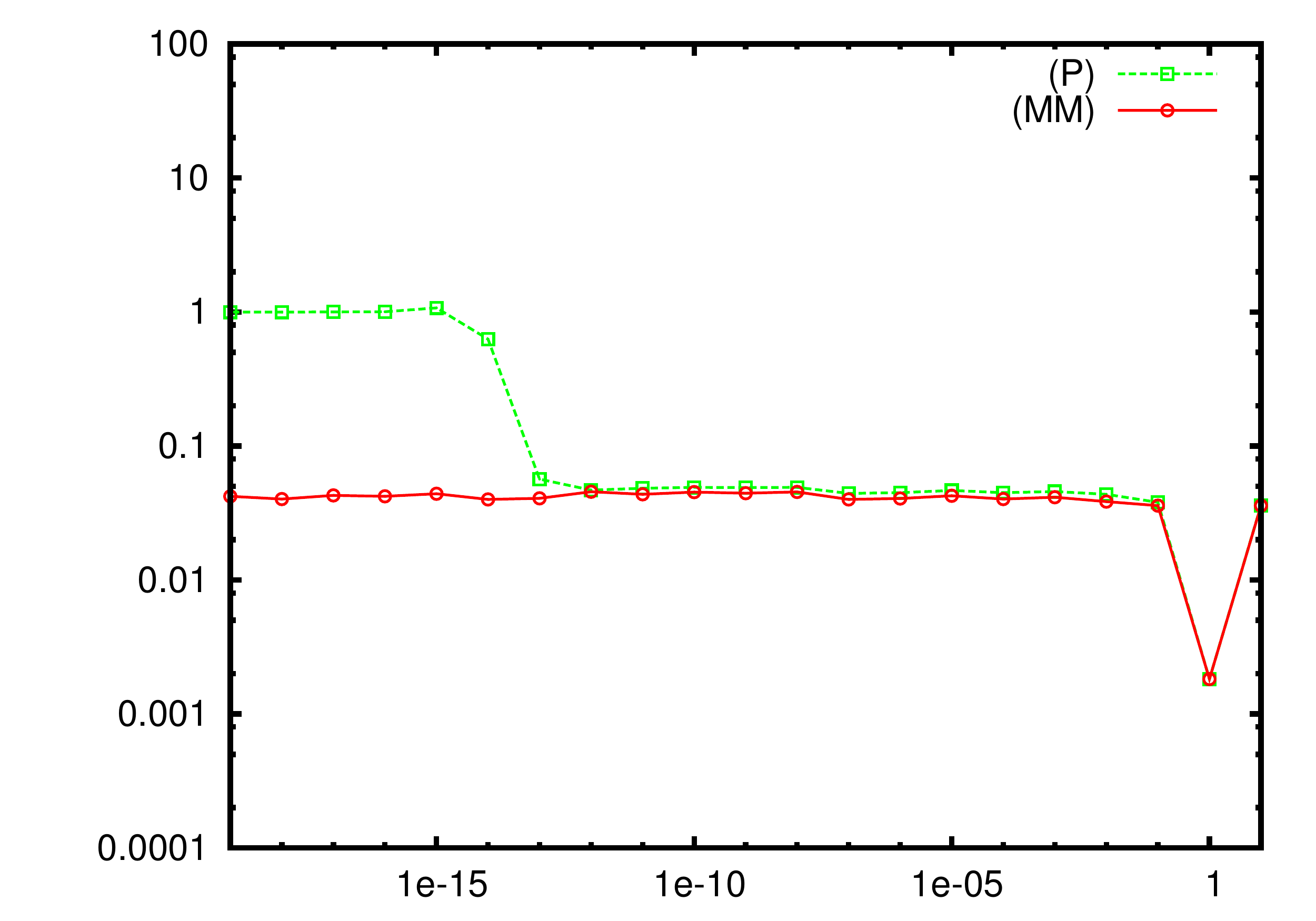}}

  \subfigure[$L^{2}$ error for a grid with $100\times 100$ points.]
  {\includegraphics[angle=-90,width=\xxxa]{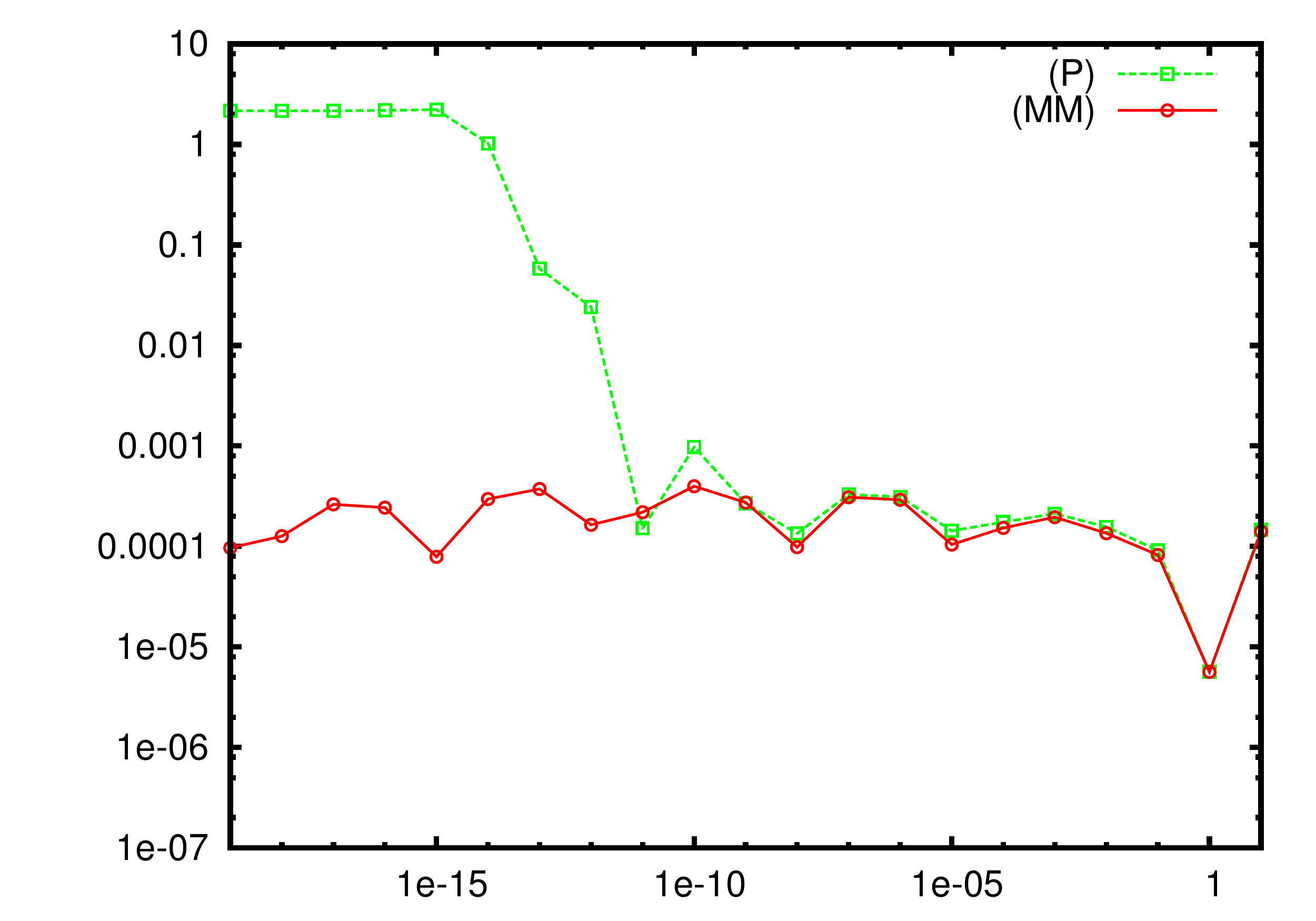}}
  \subfigure[$H^{1}$ error for a grid with $100\times 100$ points.]
  {\includegraphics[angle=-90,width=\xxxa]{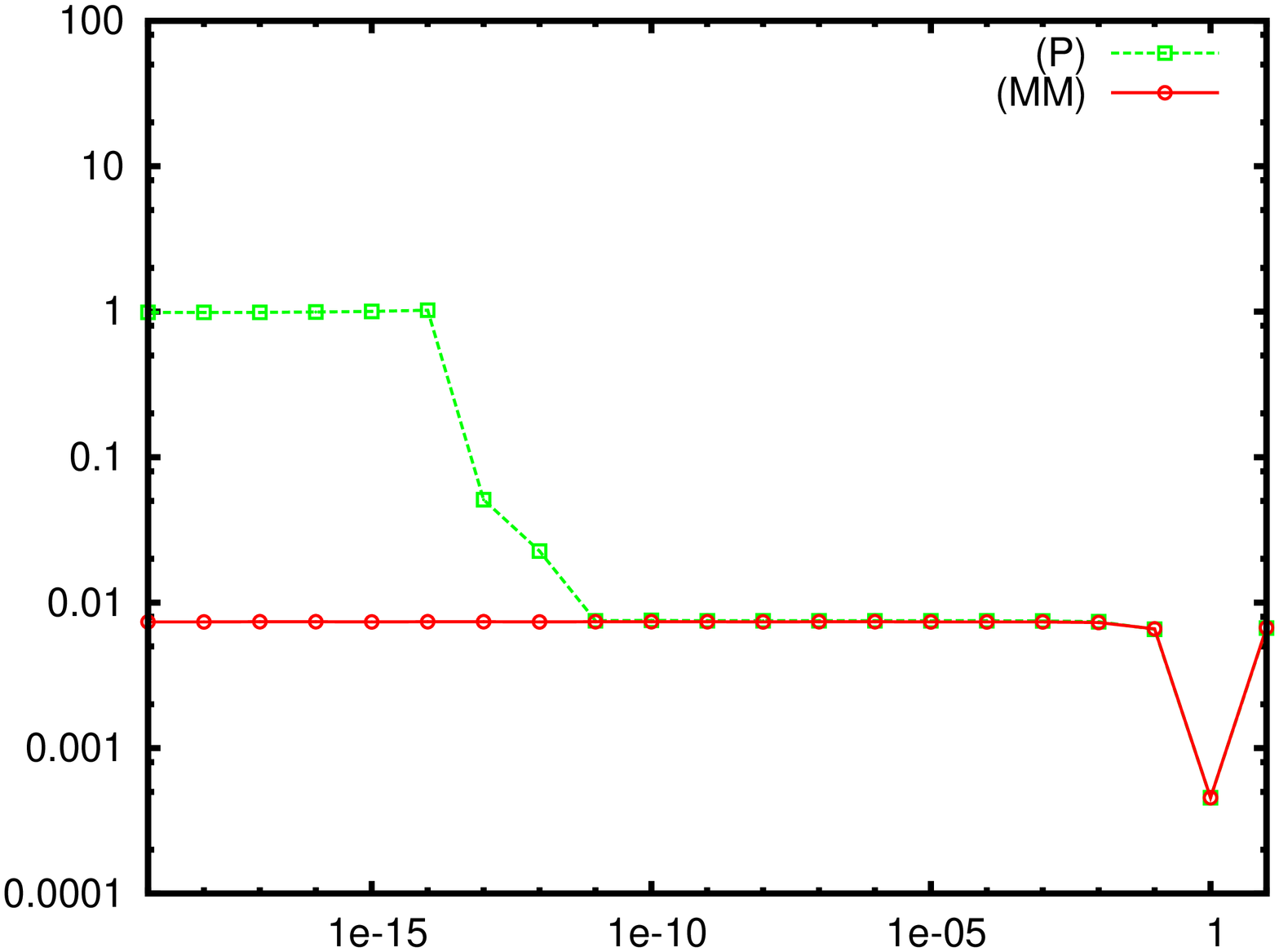}}

  \subfigure[$L^{2}$ error for a grid with $200\times 200$ points.]
  {\includegraphics[angle=-90,width=\xxxa]{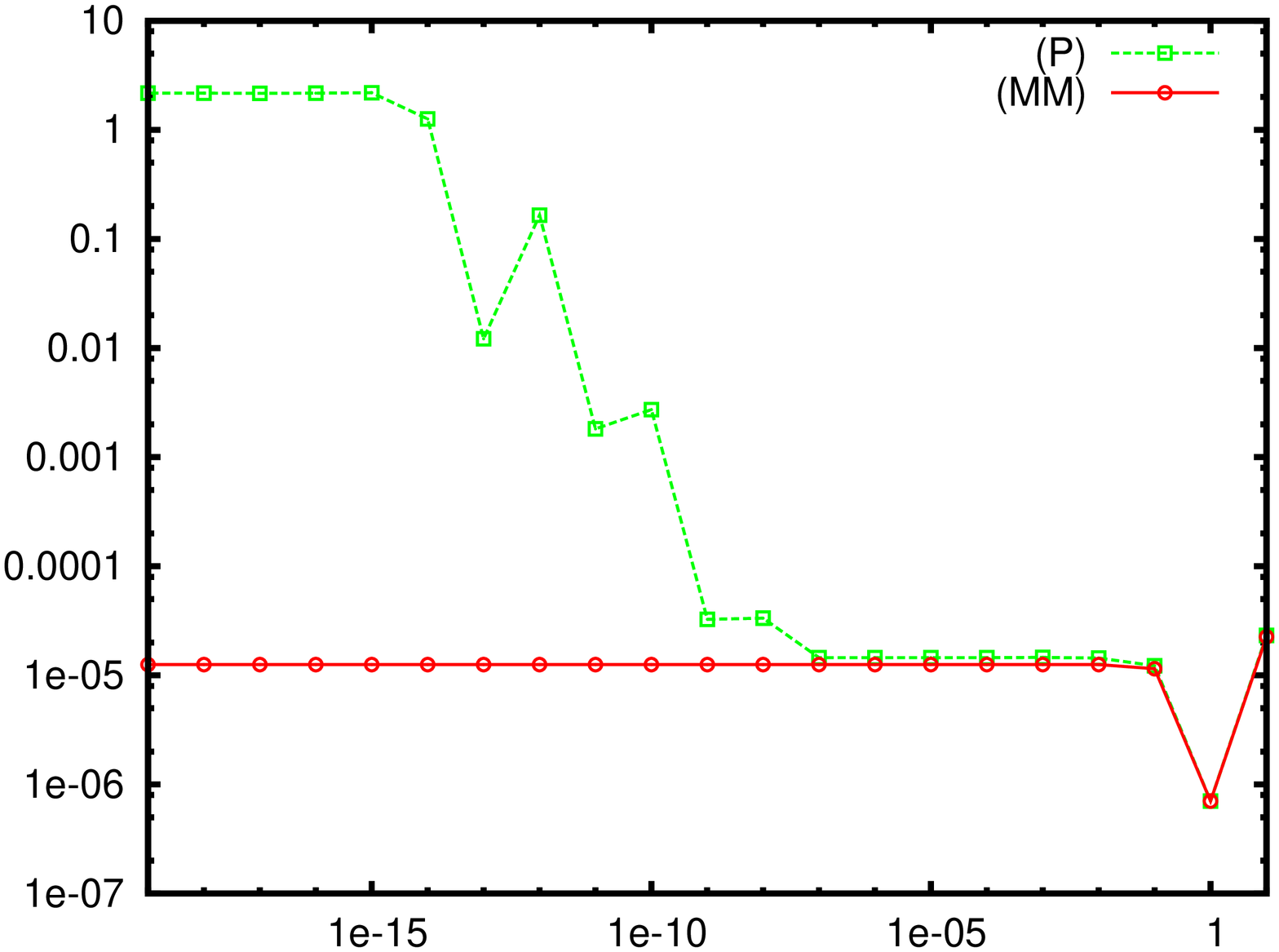}}
  \subfigure[$H^{1}$ error for a grid with $200\times 200$ points.]
  {\includegraphics[angle=-90,width=\xxxa]{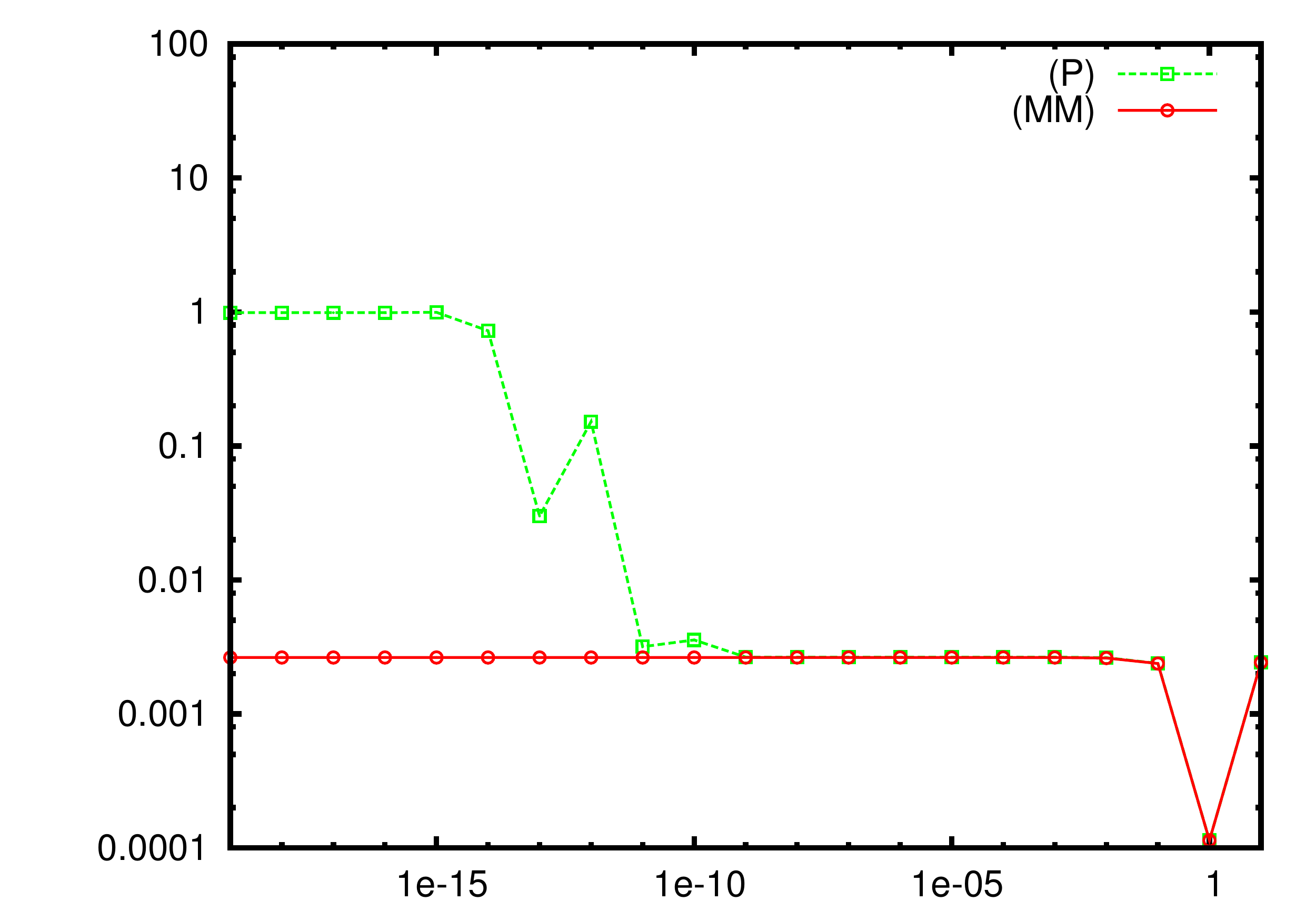}}

  \caption{Relative $L^{2}$ (left column) and $H^{1}$ (right column)
    errors between the exact solution $u^{\varepsilon }$ and the
    computed numerical solution $u_M$ (MM), $u_P$ (P) for the test
    case with constant $b$ and variable $\eps$. The error is plotted
    as a function of the parameter $\eps_{min}$ and for three
    different mesh-sizes.}
  \label{fig:error_ev_bc}
\end{figure}

\subsection{2D test case, variable $\varepsilon$, uniform and
  aligned $b$-field}

Let us consider a test case, where $\eps$ is close to $1$ in one part of
the computational domain and close to some fixed parameter
$\eps_{min}$ in the remaining part. The anisotropy varies smoothly and
changes its value in a relatively narrow transition region.
We set
\begin{gather}
  \eps(x,y) = \frac{1}{2} 
  \left[
    1+\tanh\left(a(x_0 - x)\right) +\eps_{min}  \left(1-\tanh\left(a(x_0 - x)\right)\right)
  \right]
  \label{eq:ev},
\end{gather}
with $a$ being a parameter which controls the width of the transition
region and $x_0$ the position of the interface. In our simulations we
set $x_0=0.25$ and $a=50$.

\begin{remark}
  One should put extreme attention in coding the $\eps(x,y)$
  function. If $\eps_{min}$ is smaller than the numerical precision,
  the term $1+\tanh\left(a(x_0 - x)\right)$ dominates and hence the
  value of $\eps_{min}$ is never reached. Instead one should replace
  $1+\tanh\left(a(x_0 - x)\right)$ by the equivalent term $\frac{2
    e^{2a(x_0-x)}}{e^{2a(x_0-x)}+1}$, which is not limited by a
  computer precision.
\end{remark}

 The solution $u^{\varepsilon }$ of (\ref{P}) is now given by
\begin{gather*}
  u^{\varepsilon } = \sin \left(\pi y \right) + \varepsilon \cos \left( 2\pi x\right)
  \sin \left(\pi y \right)
\end{gather*}
and let the force term $f$ be calculated accordingly.

As in the previous examples we solve the problem using $\mathbb
Q_2$-FEM. We compare the results of the Micro-Macro
reformulation with the Singular Perturbation model. The simulation
results are presented on Figure \ref{fig:error_ev_bc}. The convergence
of the MM-scheme is given in the Table \ref{tab:conv_ev}. Note that
for small mesh sizes and small values of $\eps_{min}$ superconvergence
occurs. This is caused by the small size of the transition region
between two $\eps$ regimes. For sufficiently small meshes the
convergence attains the optimal rate. Similarly as in the constant
$\eps$ test cases, the Micro-Macro scheme is capable to
produce the accurate results regardless of the anisotropy strength.

\begin{table}
  \centering
  \begin{tabular}{|c||c|c||c|c|}
    \hline
    \multirow{2}{*}{$h$}  
    & \multicolumn{2}{|c||}{\rule{0pt}{2.5ex}$\eps_{min}=1$} 
    & \multicolumn{2}{|c|}{$\eps_{min}=10^{-20}$} \\
    \cline{2-5} 
    &\rule{0pt}{2.5ex}
    $L^{2}$-error& $H^{1}$-error& $L^{2}$-error & $H^{1}$-error\\
    \hline
    \hline\rule{0pt}{2.5ex}
    0.1 &
    $5.7\times 10^{-3}$ &
    $1.86\times 10^{-1}$ &
    $2.94$ &
    $9.3\times 10^{-1}$ 
    \\
    \hline\rule{0pt}{2.5ex}
    0.05 &
    $7.3\times 10^{-4}$ &
    $4.7\times 10^{-2}$ &
    $2.12\times 10^{-1}$ &
    $7.5\times 10^{-1}$ 
    \\
    \hline\rule{0pt}{2.5ex}
    0.025 &
    $9.1\times 10^{-5}$ &
    $1.18\times 10^{-2}$ &
    $1.74\times 10^{-2}$ &
    $1.51\times 10^{-1}$ 
    \\
    \hline\rule{0pt}{2.5ex}
    0.0125 &
    $1.14\times 10^{-5}$ &
    $2.96\times 10^{-3}$ &
    $3.14\times 10^{-4}$ &
    $6.3 \times 10^{-2}$ 
    \\
    \hline\rule{0pt}{2.5ex}
    0.00625 &
    $1.43\times 10^{-6}$ &
    $7.4 \times 10^{-4}$ &
    $2.09\times 10^{-5}$ &
    $1.29\times 10^{-2}$ 
    \\
    \hline\rule{0pt}{2.5ex}
    0.003125 &
    $1.78\times 10^{-7}$ &
    $1.85\times 10^{-4}$ &
    $2.60\times 10^{-6}$ &
    $3.2 \times 10^{-3}$ 
    \\
    \hline\rule{0pt}{2.5ex}
    0.0015625 &
    $2.23\times 10^{-8}$ &
    $4.6\times 10^{-5}$ &
    $3.3\times 10^{-7}$ &
    $8.1\times 10^{-4}$ 
    \\
    \hline
  \end{tabular}
  \caption{The absolute error of $u$ in $L^{2}$ and $H^{1}$-norms 
    for different mesh sizes and $\eps_{min} =1$ resp. $\eps_{min}
    =10^{-20}$ using the Micro-Macro scheme (MM) for the
    variable $\eps$ and constant $b$ test case. 
  }
  \label{tab:conv_ev}
\end{table}

\def\xxxa{0.45\textwidth}
\begin{figure}[!ht] 
  \centering
  \subfigure[$L^{2}$ error for a grid with $50\times 50$ points.]
  {\includegraphics[angle=-90,width=\xxxa]{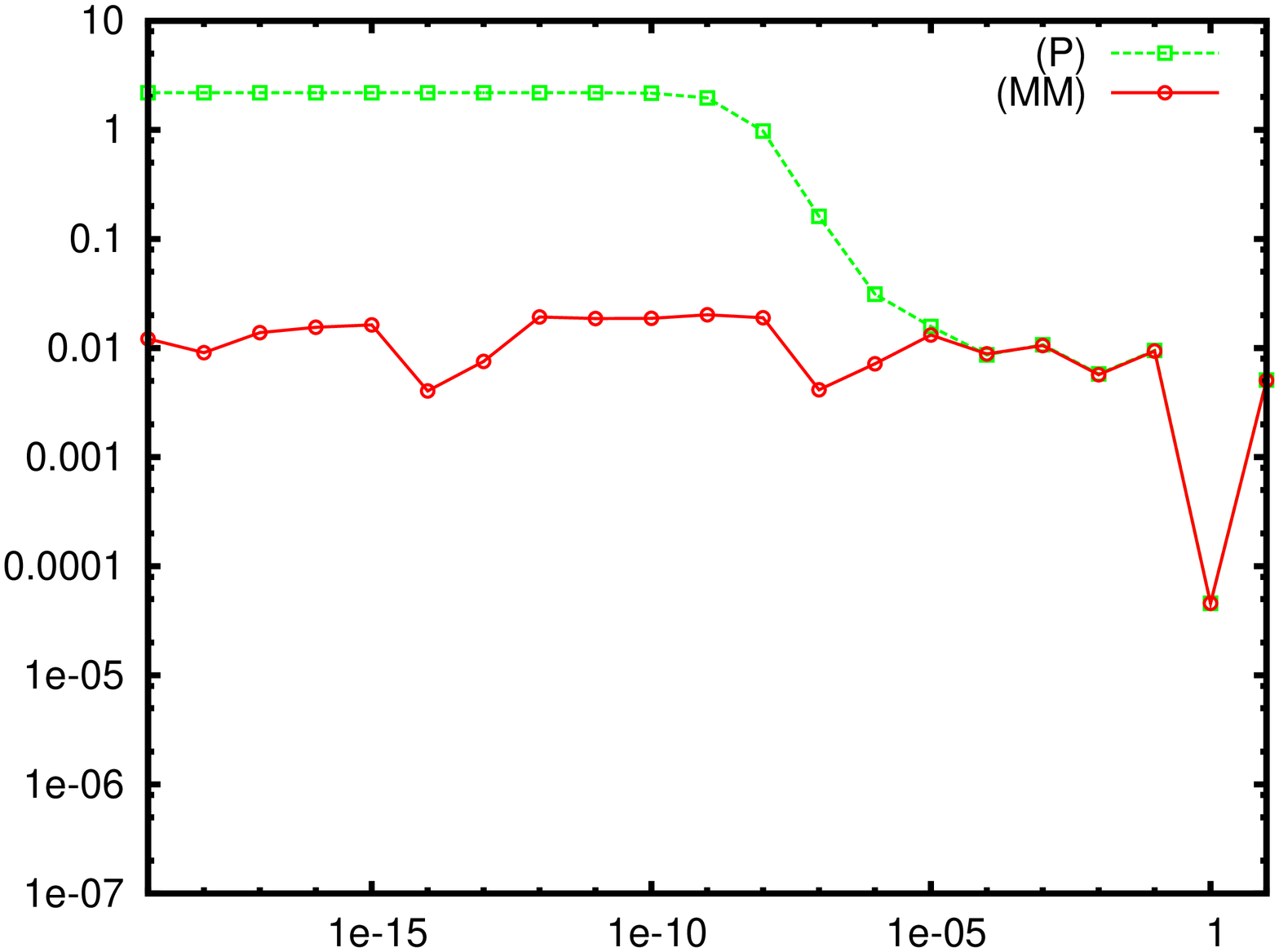}}
  \subfigure[$H^{1}$ error for a grid with $50\times 50$ points.]
  {\includegraphics[angle=-90,width=\xxxa]{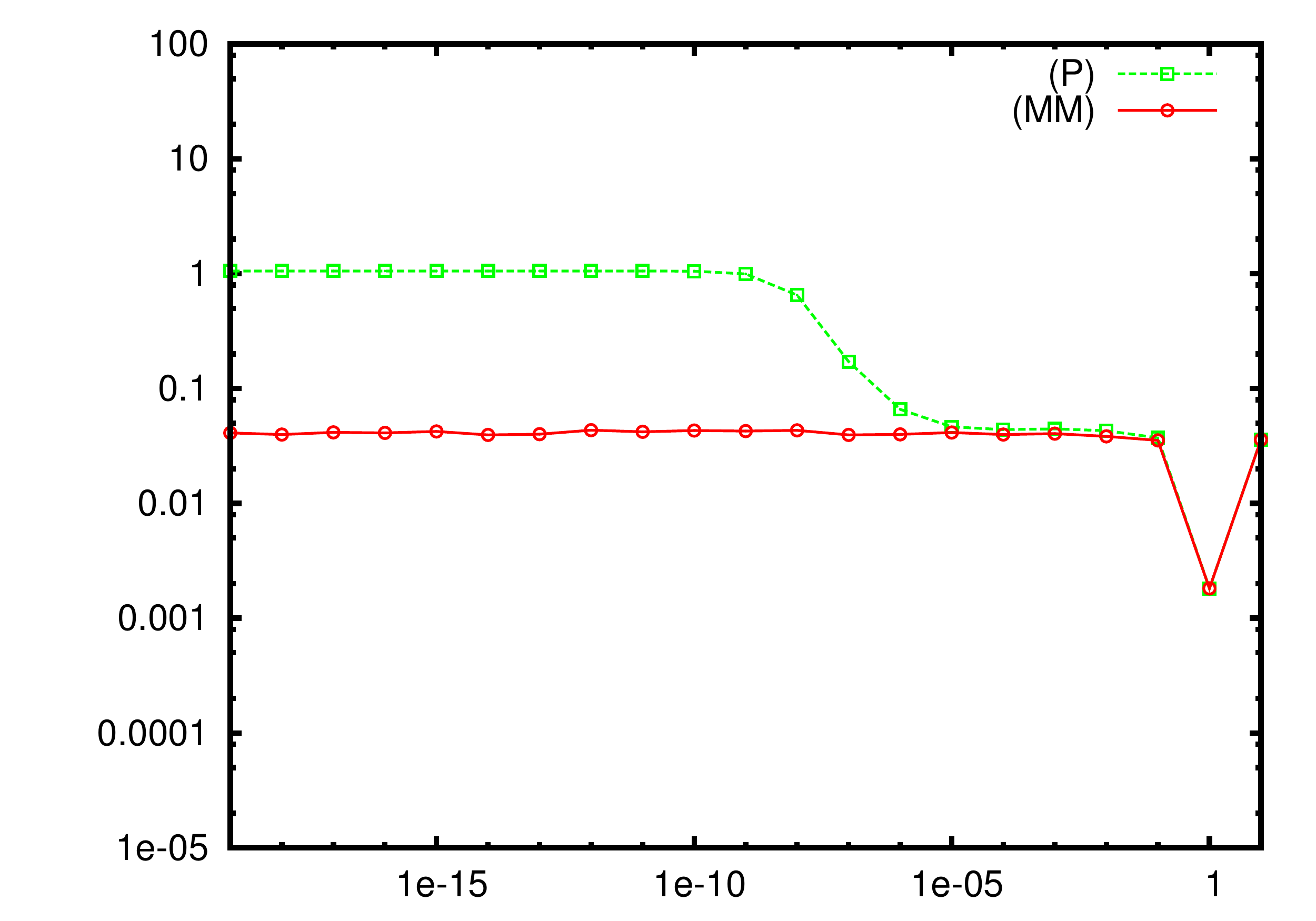}}

  \subfigure[$L^{2}$ error for a grid with $100\times 100$ points.]
  {\includegraphics[angle=-90,width=\xxxa]{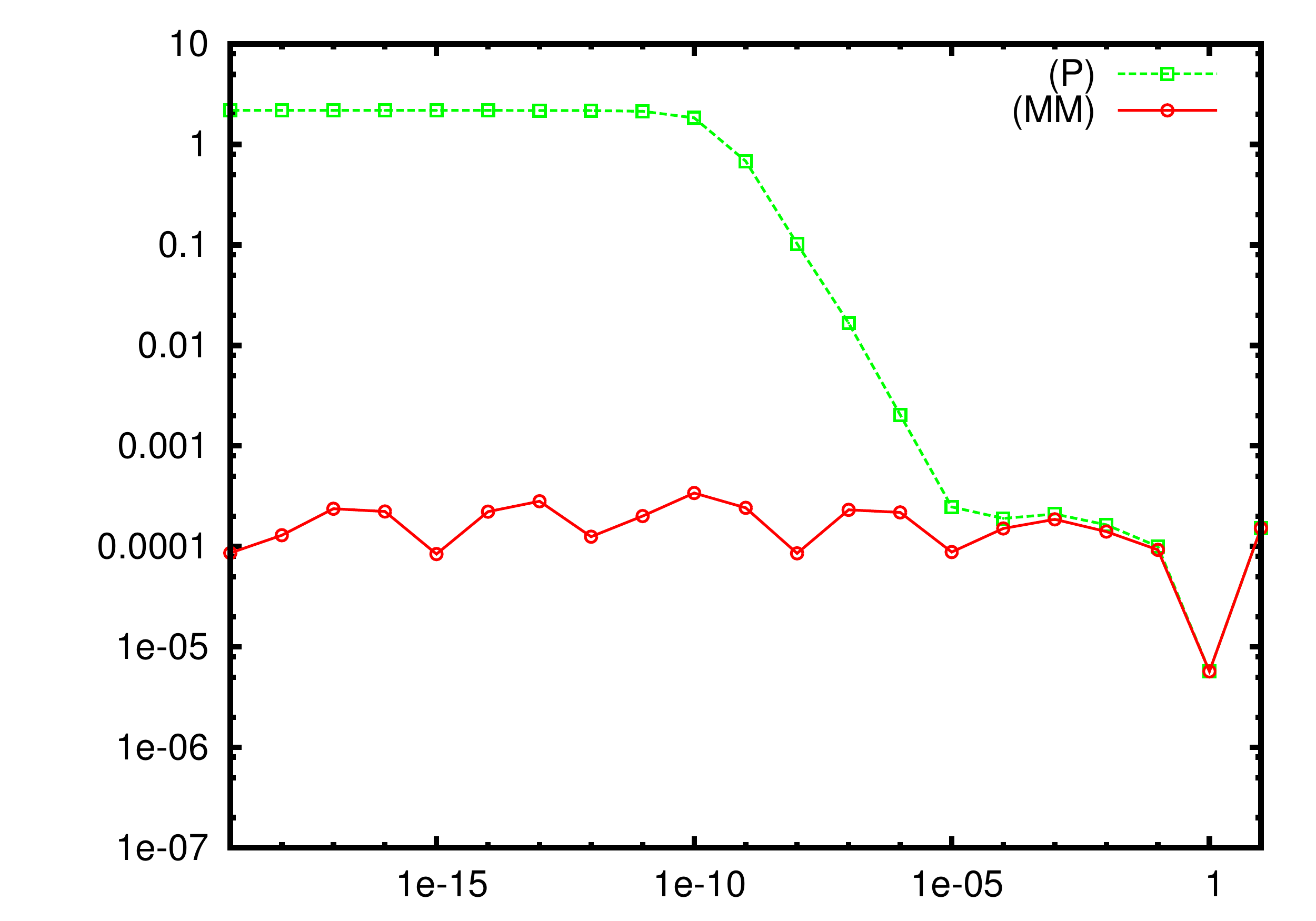}}
  \subfigure[$H^{1}$ error for a grid with $100\times 100$ points.]
  {\includegraphics[angle=-90,width=\xxxa]{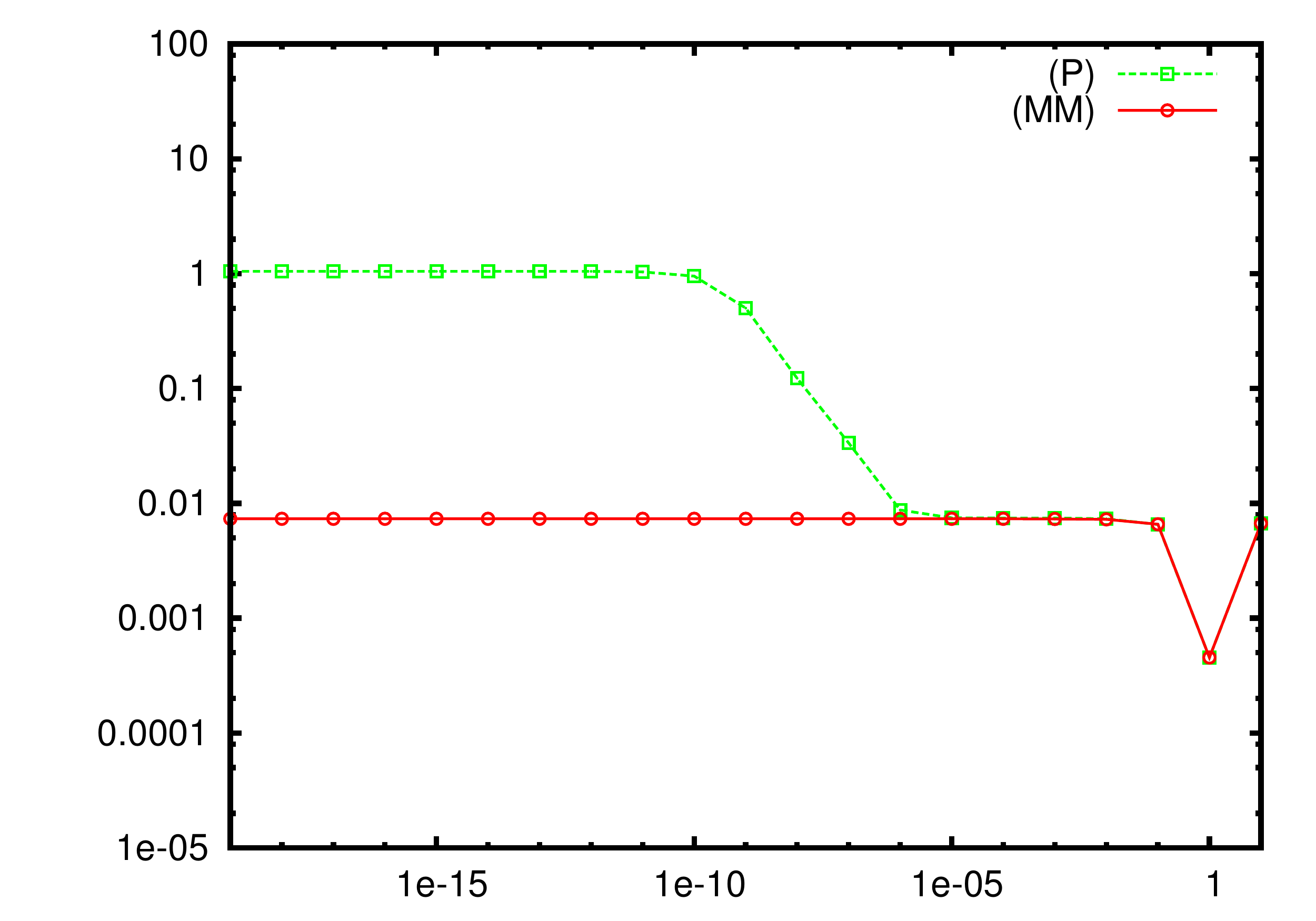}}

  \subfigure[$L^{2}$ error for a grid with $200\times 200$ points.]
  {\includegraphics[angle=-90,width=\xxxa]{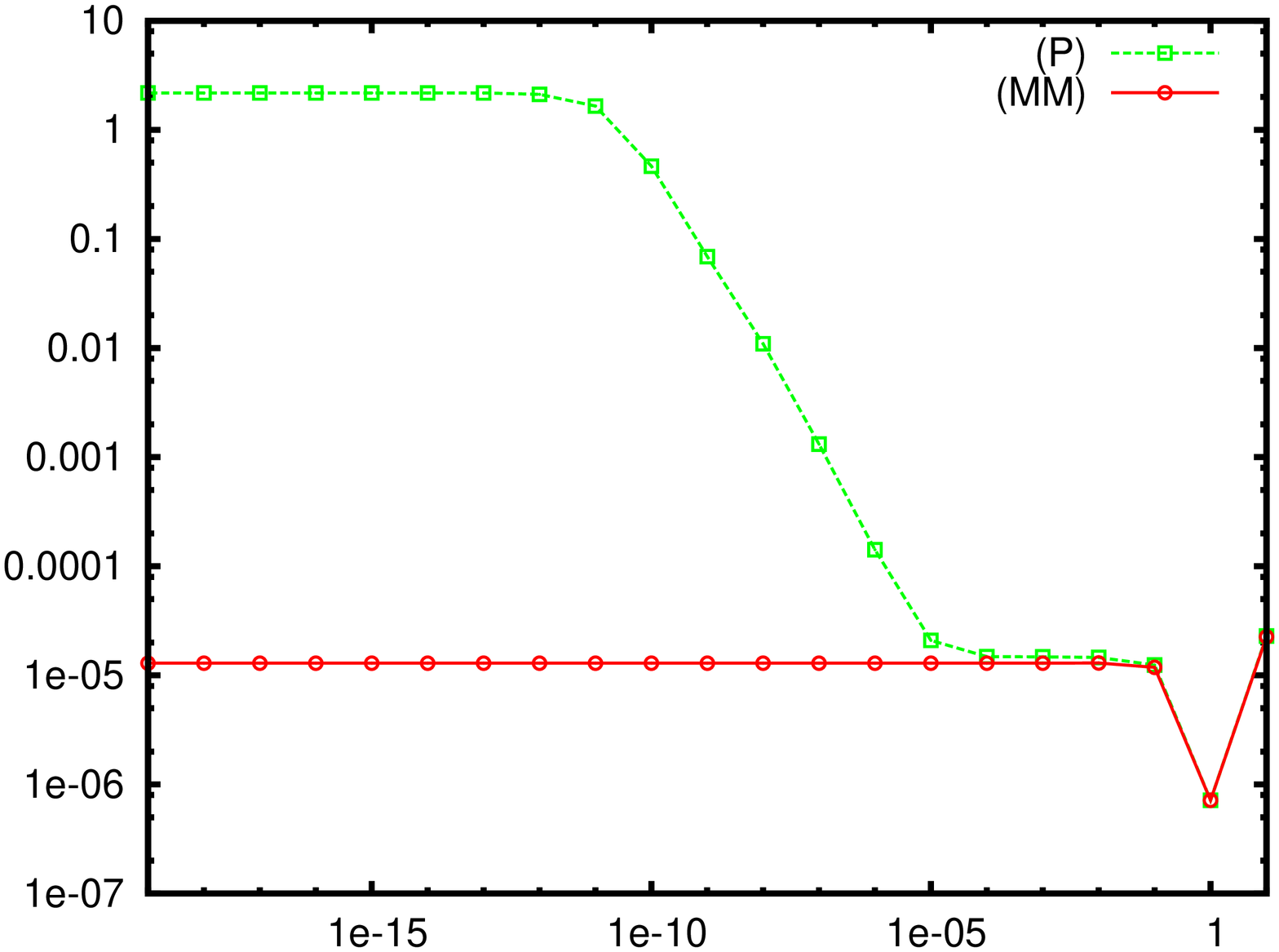}}
  \subfigure[$H^{1}$ error for a grid with $200\times 200$ points.]
  {\includegraphics[angle=-90,width=\xxxa]{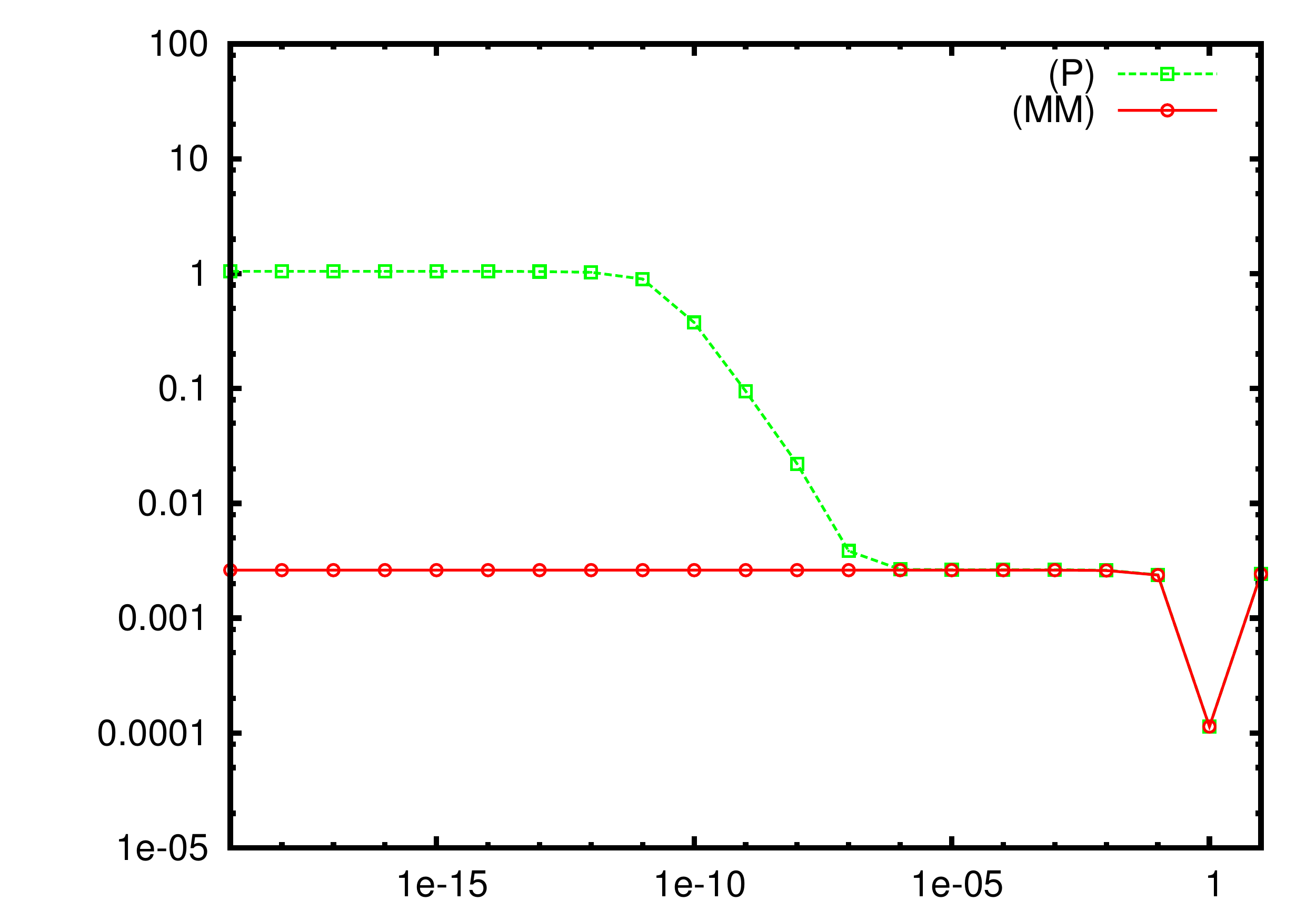}}

  \caption{Relative $L^{2}$ (left column) and $H^{1}$ (right column)
    errors between the exact solution $u^{\varepsilon }$ and the
    computed numerical solution $u_M$ (MM) resp. $u_P$ (P) for the test
    case with variable $b$ and $\eps$. The error is plotted as a
    function of the parameter $\eps_{min}$ and for three different
    mesh-sizes.}
  \label{fig:error_ev_bv}
\end{figure}

\subsection{2D test case, variable $\varepsilon$, non-uniform and non-aligned $b$-field}

In this section we choose the same test case as in  Section
\ref{sec:tc_bv} but again the epsilon varies in the computational
domain and is defined by (\ref{eq:ev}). The analytical solution to the
problem is 
\begin{gather*}
  u^{\varepsilon } = \sin \left(\pi y +\alpha (y^2-y)\cos (\pi
    x) \right) + \varepsilon \cos \left( 2\pi x\right) \sin \left(\pi
    y \right)
\end{gather*}
where $\alpha=2$ and the force term $f$ is calculated accordingly.

The results are shown on Figure \ref{fig:error_ev_bv}. Again, contrary
to the Singular Perturbation formulation, the Micro-Macro Asymptotic-Preserving
formulation is capable to produce reliable numerical results
regardless of the anisotropy strength. Similarly, the optimal
convergence rate is attained when the mesh size is small enough to capture
the $\eps$ transition.

\section{Conclusion} \label{SEC5}
The construction of the here introduced Micro-Macro based
Asymptotic-Preserving scheme for the resolution of an anisotropic
diffusion equation was based on a different reformulation of the
initial Singular Perturbation problem as compared to the earlier work
\cite{DDLNN}. The advantages of the new MM-verison were shown
numerically, in particular the considerable gain in simulation time
and the simplicity to treat variable $\eps$-intensities within the
domain. The rigorous numerical analysis of both AP-reformulations will
be the aim of a futur work, in particular the AP-property shall be
shown, i.e. the uniform convergence of the scheme with respect to
$\eps$.
\section*{Acknowledgments}
We acknowledge the support of the ``Federation de Recherche sur la
Fusion Magn\'etique'' Euroatom and CEA, under the grand ``APPLA'', of
the University Paul Sabatier Scientific Council under the grand
``Mositer'' and of the ''Agence Nationale pour la Recherche (ANR)'' in
the frame of the contract ''BOOST''.

\medskip

\bibliographystyle{abbrv}
\bibliography{bib_aniso}

\end{document}